\documentclass[a4paper,reqno,10pt]{amsart}

\usepackage{epsf,graphicx,epsfig}
\usepackage{amscd}
\usepackage{amsmath,latexsym,amssymb,amsthm}
\usepackage[nospace,noadjust]{cite}
\usepackage{textcomp}
\usepackage{setspace,cite}
\usepackage{mathrsfs,amsmath}
\usepackage{lscape,fancyhdr,fancybox}
\usepackage{stmaryrd}
\usepackage[all,cmtip]{xy}
\usepackage{tikz}
\usepackage{cancel}
\usetikzlibrary{shapes,arrows,decorations.markings}
\usepackage{graphicx}

\usepackage{color}
\usepackage{bbm}
\usepackage{xcolor}
\usepackage{url}
\usepackage{enumerate}
\usepackage[mathscr]{euscript}

  \theoremstyle{plain}

\swapnumbers
    \newtheorem{thm}{Theorem}[section]
    \newtheorem{prop}[thm]{Proposition}

    \newtheorem{subsec}[thm]{}
\theoremstyle{definition}
    \newtheorem{defn}[thm]{Definition}
        \newtheorem{remark}[thm]{Remark}
\theoremstyle{remark}

\title{}
\author{}
\date{}
\usepackage{amssymb}

\usepackage{hyperref}
\hypersetup{
colorlinks,
citecolor=blue,
filecolor=black,
linkcolor=blue,
urlcolor=black
}

\begin{document}

\title[Rota-Baxter family algebras]{Deformations and homotopy theory for Rota-Baxter family algebras}

\author{Apurba Das}
\address{Department of Mathematics,
Indian Institute of Technology, Kharagpur 721302, West Bengal, India.}
\email{apurbadas348@gmail.com, apurbadas348@maths.iitkgp.ac.in}

\begin{abstract}
The concept of Rota-Baxter family algebra is a generalization of Rota-Baxter algebra. It appears naturally in the algebraic aspects of renormalizations in quantum field theory. Rota-Baxter family algebras are closely related to dendriform family algebras. In this paper, we first construct an $L_\infty$-algebra whose Maurer-Cartan elements correspond to Rota-Baxter family algebra structures. Using this characterization, we define the cohomology of a given Rota-Baxter family algebra. As an application of our cohomology, we study formal and infinitesimal deformations of a given Rota-Baxter family algebra. Next, we define the notion of a homotopy Rota-Baxter family algebra structure on a given $A_\infty$-algebra. We end this paper by considering the homotopy version of dendriform family algebras and their relations with homotopy Rota-Baxter family algebras.
\end{abstract}

\maketitle



\medskip

\begin{center}



\end{center}



\thispagestyle{empty}

\tableofcontents


\medskip


\section{Introduction}
\subsection{Rota-Baxter algebras and dendriform algebras} Rota-Baxter operators are an algebraic abstraction of the classical integral operator. They first appeared in the work of G. Baxter in his study of the fluctuation theory in probability \cite{baxter}. Subsequently, such operators were studied by G.-C. Rota \cite{rota}, P. Cartier \cite{cartier} and F. V. Atkinson \cite{atkinson}, among others. Given an associative algebra $A$, a linear map $R:A \rightarrow A$ is said to be a {\bf Rota-Baxter operator} on $A$ if
\begin{align}\label{rb-identity}
    R(a) \cdot R(b) = R \big( R(a) \cdot b  \! ~+~ \! a \cdot R(b)   \big), \text{ for } a, b \in A.
\end{align}
Here dot $\cdot$ denotes the associative multiplication on the algebra $A$. A pair $(A,R)$ consisting of an associative algebra $A$ and a Rota-Baxter operator $R:A \rightarrow A$ is called a {\bf Rota-Baxter algebra}. In the last twenty years, Rota-Baxter algebras got very much attention due to their intimate relations with pre-Lie algebras, dendriform algebras, shuffle algebras, splitting of operads, associative Yang-Baxter equations, infinitesimal bialgebras and renormalizations in quantum field theory \cite{aguiar,aguiar4,aguiar3,guo-keigher,bai-olivia,connes,fard-guo2}. In \cite{uchino}, K. Uchino introduced a notion of relative Rota-Baxter operator in his study of Poisson structures in the noncommutative set-up. Let $A$ be an associative algebra and $M$ be an $A$-bimodule. A {\bf relative Rota-Baxter operator} on $M$ over the algebra $A$ is a linear map $R:M\rightarrow A$ satisfying
\begin{align*}
    R(u) \cdot R(v) = R \big( R(u) \cdot_M v  \! ~+~ \! u \cdot_M R(v)   \big), \text{ for } u,v \in M,
\end{align*}
where $~\cdot_M ~$ denotes both the left and right $A$-actions on $M.$ It follows that any Rota-Baxter operator on $A$ can be regarded as a relative Rota-Baxter operator on the adjoint $A$-bimodule $A$. A triple $(A,M, R)$ consisting of an associative algebra $A$, an $A$-bimodule $M$ and a relative Rota-Baxter operator $R:M \xrightarrow{} A$ is called a {\bf relative Rota-Baxter algebra}. It has been shown in \cite{uchino} that relative Rota-Baxter algebras are closely related to dendriform algebras as introduced by J.-L. Loday \cite{loday}. Recall that a {\bf dendriform algebra} is a triple $(D, \prec, \succ)$ consisting of a vector space $D$ equipped with two linear maps $\prec, \succ : D \otimes D \xrightarrow{} D$ satisfying
\begin{align*}
    (x \prec y ) \prec z =~& x \prec (y \prec z + y \succ z),\\
      (x \succ y ) \prec z =~&  x \succ (y  \prec z),\\
      (x \prec y + x \succ y) \succ z =~& x \succ (y \succ z), \text{ for } x, y, z \in D.
\end{align*}
If $(A, M, R)$ is a relative Rota-Baxter algebra then $(M, \prec, \succ)$ is a dendriform algebra, where 
\begin{align}\label{induced-dend-family}
u \prec v := u \cdot_M R(v) \quad \text{  and } \quad u \succ v := R(u) \cdot_M v, \text{ for } u, v \in M.
\end{align}
See \cite{fard-guo,ebrahimi,gubarev} for more details about (relative) Rota-Baxter algebras and their relations with dendriform algebras.

\medskip

\subsection{Cohomology and deformation theory for algebraic structures} Algebraic structures are better understood by knowing some invariants (e.g. homology, cohomology, $K$-groups) associated with them. In \cite{hoch} G. Hochschild first introduced the cohomology theory of associative algebras in his study of extensions of algebras. This cohomology is now known as Hochschild cohomology. In 1964, M. Gerstenhaber \cite{gers} first considered algebraic deformation theory for associative algebras and showed that such deformations are governed by the Hochschild cohomology theory. Subsequently, cohomology and deformation theories have been generalized for various other algebraic structures, including Lie algebras, Leibniz algebras, Poisson algebras and many others \cite{nij-ric,nij-ric2,bala-leib,penkava}. These theories are not only defined for various types of algebras, rather such theories are also defined for various operators on algebras. In \cite{gers-sch} M. Gerstenhaber and S. Schack considered cohomology and deformation theories for algebra morphisms. Later, they were extensively studied by Y. Fr\'{e}gier, M. Markl, D. Yau and M. Zambon \cite{diagram,fre-zam}. Recently, the authors in \cite{tang,laza,das-rota,das-mishra} defined cohomology and deformation theories for Rota-Baxter operators on Lie algebras and associative algebras. The cohomological relations between Rota-Baxter algebras and dendriform algebras are described in \cite{das-mishra2}. 

\medskip  

\subsection{Rota-Baxter family algebras (cohomology and deformations)} Family algebraic structures indexed by a semigroup first appeared in the work of K. Ebrahimi-Fard, J. Gracia-Bondia and F. Patron in their study of renormalizations in quantum field theory \cite{ebrahimi-gracia} (see also \cite{kreimer-panzer}). In these algebras, the usual operations are replaced by a family of operations labelled by the elements of a semigroup $\Omega$.  The notion of Rota-Baxter family algebra was first introduced in \cite{ebrahimi-gracia} as the `family version' of a Rota-Baxter algebra. More precisely, a Rota-Baxter family algebra is a pair $(A, \{ R_\alpha \}_{\alpha \in \Omega})$ consisting of an associative algebra $A$ and a collection $\{ R_\alpha : A \rightarrow A \}_{\alpha \in \Omega}$ of linear maps labelled by the elements of a semigroup $\Omega$, satisfying (\ref{rb-family-identity}) (see Definition \ref{rb-family-defn}) that generalize the usual Rota-Baxter identity (\ref{rb-identity}). Recently, Rota-Baxter family algebras and dendriform family algebras (family version of dendriform algebras) are extensively studied by L. Foissy, X. Gao, L. Guo, D. Manchon and Y. Zhang among others \cite{foissy,guo-guo,zgao,zgaom}. In \cite{das-ns} the author introduced the notion of a relative Rota-Baxter family algebra and find various relations with dendriform family algebras. He also introduced a cohomology associated with the collection of operators $\{ R_\alpha \}_{\alpha \in \Omega}$. However, this cohomology has not taken the underlying algebra and the bimodule into the account. Therefore, this cohomology controls the deformation theory of the collection of operators  $\{ R_\alpha \}_{\alpha \in \Omega}$ but not the whole data of a (relative) Rota-Baxter family algebra.

\medskip

In this paper, our first aim is to define the cohomology of relative Rota-Baxter family algebra which takes the underlying algebra and the bimodule into the consideration. To define this cohomology, we first construct a suitable $L_\infty$-algebra that characterizes relative Rota-Baxter family algebras as its Maurer-Cartan elements. The cohomology of a given relative Rota-Baxter family algebra is defined to be the cohomology induced by the corresponding Maurer-Cartan element. As a particular case, we obtain the cohomology of a Rota-Baxter family algebra. Next, we study formal one-parameter deformations and infinitesimal deformations of a relative Rota-Baxter family algebra as applications of our cohomology. Among all results, we show that the set of all equivalence classes of infinitesimal deformations has a bijection with the second cohomology group of the relative Rota-Baxter family algebra.

\medskip

\subsection{Homotopy Rota-Baxter family algebras and Dend$_\infty$-family algebras} The notion of $A_\infty$-algebras first appeared in the work of J. Stasheff in his study of topological loop spaces \cite{stasheff}. The concept of $A_\infty$-algebras and $L_\infty$-algebras are respectively homotopy versions of associative and Lie algebras \cite{lada-markl,lada-stasheff,keller,getzler}. In \cite{laza} the authors first considered homotopy Rota-Baxter operators on Lie algebras and find their relations with homotopy pre-Lie algebras. Subsequently, this study was generalized in the context of associative algebras in \cite{das-mishra,das-mishra2}. More precisely, the authors have introduced and studied homotopy Rota-Baxter algebras and found various connections with Dend$_\infty$-algebras (homotopy dendriform algebras).

\medskip

Another aim of this paper is to define the notion of a homotopy (relative) Rota-Baxter family algebra. We first construct a suitable $L_\infty$-algebra that motivates us to define homotopy Rota-Baxter family algebras. We also introduce Dend$_\infty$-family algebras as the homotopy version of dendriform family algebras. We observe that a Dend$_\infty$-family algebra gives rise to an $\Omega$-$A_\infty$-algebra (also called $A_\infty$-algebra relative to the semigroup $\Omega$). This generalizes the fact that a dendriform family algebra induces an associative algebra relative to $\Omega$ (see Aguiar \cite{aguiar-dend}). Finally, we show that a particular class of homotopy relative Rota-Baxter family algebras (called `strict' homotopy relative Rota-Baxter family algebras) induce Dend$_\infty$-family algebras. This is a generalization of the construction (\ref{induced-dend-family}) in the homotopy case. On the other hand, we show that any Dend$_\infty$-family algebra gives rise to a strict homotopy relative Rota-Baxter family algebra so that the induced Dend$_\infty$-family algebra coincides with the given one.

\medskip

\subsection{Organization of the paper} We organize the paper as follows. In section \ref{sec2}, we recall (relative) Rota-Baxter family algebras and some necessary background on $L_\infty$-algebras. Our aim in section \ref{sec3} is to construct the $L_\infty$-algebra that characterizes relative Rota-Baxter family algebras as its Maurer-Cartan elements. This characterization allows us to define the cohomology of a (relative) Rota-Baxter family algebra in section \ref{sec4}. As an application of cohomology, in section \ref{sec5}, we study formal one-parameter deformations and infinitesimal deformations of a relative Rota-Baxter family algebra. Finally, in section \ref{sec6}, we introduce and study homotopy relative Rota-Baxter family algebras and Dend$_\infty$-family algebras.

\medskip

Throughout the paper, all vector spaces, linear maps and unadorned tensor products are over a field ${\bf k}$ of characteristic $0$.

\section{Background on Rota-Baxter family algebras and $L_\infty$-algebras}\label{sec2}
In this section, we first recall Rota-Baxter family algebras, relative Rota-Baxter family algebras and their relations with dendriform family algebras. Next, we also recall some necessary background on $L_\infty$-algebras. Our main references are \cite{das-mishra,lada-markl,das-ns,zgao,zgaom}.

\subsection{Rota-Baxter family algebras and dendriform family algebras}
Here we collect some definitions and basic results about (relative) Rota-Baxter family algebras and dendriform family algebras.
\begin{defn}\label{rb-family-defn}
A {\bf Rota-Baxter family algebra} is a pair $(A, \{ R_\alpha \}_{\alpha \in \Omega})$ consisting of an associative algebra $A$ and a collection $\{ R_\alpha : A \rightarrow A \}_{\alpha \in \Omega}$ of linear maps labelled by the elements of $\Omega$ that satisfy
\begin{align}\label{rb-family-identity}
R_\alpha (a) \cdot R_\beta (b) = R_{\alpha \beta} \big(  R_\alpha (a) \cdot b ~+~ a \cdot R_\beta (b)  \big), \text{ for } a, b \in A, \alpha , \beta \in \Omega.
\end{align}
\end{defn}

\medskip

Let $A$ be an associative algebra. An {\bf $A$-bimodule} is a vector space $M$ equipped with two linear maps (called the left and right $A$-actions, respectively) $l : A \otimes M \rightarrow M, ~(a \otimes u) \mapsto a \cdot_M u$ and $r : M \otimes A \rightarrow M, ~ (u \otimes a) \mapsto u \cdot_M a$ satisfying
\begin{align*}
(a \cdot b) \cdot_M u = a \cdot_M (b \cdot_M u), \quad (a \cdot_M u) \cdot_M b = a \cdot_M (u \cdot_M b) ~~~ ~~ \text{ and } ~~ ~~~ (u \cdot_M a) \cdot_M b = u \cdot_M (a \cdot b),
\end{align*}
for $a, b \in A$, $u \in M$. We often denote an $A$-bimodule as above by $(M, l, r)$ or simply by $M$. It is easy to see that any associative algebra $A$ can be regarded as an $A$-bimodule, where both the left and right $A$-actions are given by the algebra multiplication map. This is called the adjoint $A$-bimodule.

\begin{defn}
A {\bf relative Rota-Baxter family algebra} is a triple $(A, M, \{ R_\alpha \}_{\alpha \in \Omega})$ consisting of an associative algebra $A$, an $A$-bimodule $M$ and a collection $\{ R_\alpha : M \rightarrow A \}_{\alpha \in \Omega}$ of linear maps satisfying
\begin{align}\label{rel-rb-fam-iden}
R_\alpha (u ) \cdot R_\beta (v) = R_{\alpha \beta} \big(  R_\alpha (u) \cdot_M v ~+~ u \cdot_M R_\beta (v)   \big), \text{ for } u, v \in M, \alpha , \beta \in \Omega.
\end{align}
\end{defn}

It is easy to see that any Rota-Baxter family algebra $(A, \{ R_\alpha \}_{\alpha \in \Omega})$ can be regarded as a relative Rota-Baxter family algebra $(A, A, \{ R_\alpha \}_{\alpha \in \Omega})$, where $A$ is equipped with the adjoint $A$-bimodule structure.

In the next, we recall dendriform family algebras which are the underlying structure of (relative) Rota-Baxter family algebras.

\begin{defn}
A {\bf dendriform family algebra} is a pair $(D, \{ \prec_\alpha, \succ_\alpha \}_{\alpha \in \Omega})$ consisting of a vector space $D$ equipped with a collection of linear maps $\{ \prec_\alpha , \succ_\alpha : D \otimes D \rightarrow D \}_{\alpha \in \Omega}$ labelled by the elements of $\Omega$ satisfying
\begin{align*}
( x \prec_\alpha y ) \prec_\beta z =~& x \prec_{\alpha \beta} (y \prec_\beta z + y \succ_\alpha z), \\
( x \succ_\alpha y ) \prec_\beta z =~&  x \succ_\alpha (y  \prec_\beta z), \\
(x \prec_\beta y + x \succ_\alpha y) \succ_{\alpha \beta} z =~& x \succ_\alpha (y \succ_\beta z), \text{ for } x, y, z \in D \text{ and } \alpha, \beta \in \Omega.
\end{align*}
\end{defn}

It follows from the above definition that any dendriform algebra $(D, \prec, \succ)$ can be considered as a dendriform family algebra $(D, \{ \prec_\alpha, \succ_\alpha \}_{\alpha \in \Omega})$, where $\prec_\alpha ~=~ \prec$ and $\succ_\alpha ~=~ \succ$, for all $\alpha \in \Omega$.

\begin{prop}
Let $(A, M, \{ R_\alpha \}_{\alpha \in \Omega})$ be a relative Rota-Baxter family algebra. Then the vector space $M$ inherits a dendriform family algebra structure with the operations
\begin{align*}
u \prec_\alpha v := u \cdot_M R_\alpha (v) ~~~~ \text{ and } ~~~~ u \succ_\alpha v := R_\alpha (u) \cdot_M v, ~\text{ for } u, v \in M, ~\alpha \in \Omega.
\end{align*}
In other words, $(M,  \{ \prec_\alpha, \succ_\alpha \}_{\alpha \in \Omega})$ is a dendriform family algebra, called the induced dendriform family algebra.
\end{prop}

Let $(D,  \{ \prec_\alpha, \succ_\alpha \}_{\alpha \in \Omega} )$ be any dendriform family algebra. Then the space $D \otimes {\bf k} \Omega$ carries a dendriform algebra structure with the operations
\begin{align*}
(x \otimes \alpha) \prec (y \otimes \beta) = (x \prec_\beta y) \otimes \alpha \beta ~~~~ \text{ and } ~~~~ (x \otimes \alpha) \succ (y \otimes \beta)  = (x \succ_\alpha y) \otimes \alpha \beta,
\end{align*}
for $x \otimes \alpha,~ y \otimes \beta \in D \otimes {\bf k} \Omega$ (see \cite{zgaom}). As a consequence, $D \otimes {\bf k} \Omega$ inherits an associative algebra structure with the multiplication (denoted by $~\cdot_\mathrm{Tot}~$) given by
\begin{align*}
(x \otimes \alpha) \cdot_\mathrm{Tot} (y \otimes \beta) := (x \otimes \alpha) \prec (y \otimes \beta) ~+~ (x \otimes \alpha) \succ (y \otimes \beta).
\end{align*}
We denote this associative algebra simply by $(D \otimes {\bf k} \Omega)_\mathrm{Tot}$. Further, the vector space $D$ carries a  $(D \otimes {\bf k} \Omega)_\mathrm{Tot}$-bimodule structure with the left and right  $(D \otimes {\bf k} \Omega)_\mathrm{Tot}$-actions
\begin{align*}
(x \otimes \alpha) \cdot_D y := x \succ_\alpha y \quad \text{ and } \quad y \cdot_D (x \otimes \alpha) := y \prec_\alpha x, \text{ for } x \otimes \alpha \in  (D \otimes {\bf k} \Omega)_\mathrm{Tot},~ y \in D.
\end{align*}
With all these notations, the tuple  $ \big( (D \otimes {\bf k} \Omega)_\mathrm{Tot}, D, \{ R_\alpha \}_{\alpha \in \Omega} \big)$ is a relative Rota-Baxter family algebra, where
\begin{align*}
R_\alpha : D \rightarrow  (D \otimes {\bf k} \Omega)_\mathrm{Tot}, ~ R_\alpha (x) =  x \otimes \alpha, \text{ for }x \in D.
\end{align*}
Moreover, the induced dendriform family algebra structure on $D$ coincides with the given one.

\medskip

\subsection{$L_\infty$-algebras} In this subsection, we recall some basic definitions and results regarding $L_\infty$-algebras. Here we follow the sign conventions of Voronov \cite{voro} which coincides with that of Lada and Markl \cite{lada-markl} by a degree shift.

\begin{defn}
An {\bf $L_\infty$-algebra} is a pair $(\mathfrak{L}, \{ l_k \}_{k \geq 1})$ consisting of a graded vector space $\mathfrak{L} = \oplus_{i \in \mathbb{Z}} \mathfrak{L}_i$ equipped with a collection of degree $1$ linear maps $\{ l_k : \mathfrak{L}^{\otimes k} \rightarrow \mathfrak{L} \}_{k \geq 1}$ satisfying

- (graded symmetry) for $k \geq 1$, $\sigma \in \mathbb{S}_k$ and homogeneous elements $x_1, \ldots, x_k \in \mathfrak{L}$,
\begin{align*}
l_k \big(  x_{\sigma (1)}, \ldots, x_{\sigma (k)}  \big) = \epsilon(\sigma)~ l_k (x_1, \ldots, x_k),
\end{align*}

- (shifted higher Jacobi identities) for $n \geq 1$ and homogeneous elements $x_1, \ldots, x_n \in \mathfrak{L}$,
\begin{align*}
\sum_{i+j=n+1} \sum_{\sigma \in \mathbb{S}_{(i, n-i)}} (-1)^\sigma ~\epsilon (\sigma) ~ l_j \big(  l_i ( x_{\sigma (1)}, \ldots, x_{\sigma (i)}  ), x_{\sigma (i+1)}, \ldots, x_{\sigma (n)}  \big) = 0.
\end{align*}
Here $\mathbb{S}_k$ denotes the set of all permutations on the set $\{ 1, \ldots, k \}$, and $\mathbb{S}_{(i, n-i)}$ is the set of all $(i, n-i)$-shuffles. The notation $\epsilon (\sigma) = \epsilon (\sigma; x_1, \ldots, x_n)$ represents the {\em Koszul sign} that appears in the graded context.
\end{defn}

Any differential graded Lie algebra (in particular, any graded Lie algebra) can be regarded as an $L_\infty$-algebra. More precisely, if $(\mathfrak{g}, [~,~], d)$ is a differential graded Lie algebra, then $({s^{-1}} \mathfrak{g}, \{ l_k \}_{k \geq 1})$ is an $L_\infty$-algebra, where 
\begin{align*}
l_1 (s^{-1} x) = s^{-1} (dx), \quad l_2 (s^{-1}x, s^{-1} y) = (-1)^{|x|} s^{-1} [x, y] \quad \text{ and } \quad l_k = 0 \text{ for } k \geq 3.
\end{align*}

Throughout this paper, we assume that all $L_\infty$-algebras are weakly filtered \cite{getzler,laza}. In other words, we assume that certain formal sums of elements of $\mathfrak{L}$ always converge to some elements in $\mathfrak{L}$. This assumption is essential while considering Maurer-Cartan elements in an $L_\infty$-algebra. 

\begin{defn}
Let $( \mathfrak{L}, \{ l_k \}_{k \geq 1}  )$ be an $L_\infty$-algebra. An element $\alpha \in \mathfrak{L}_0$ is said to be a {\bf Maurer-Cartan element} of the $L_\infty$-algebra if the element $\alpha$ satisfies
\begin{align*}
\sum_{k=1}^\infty \frac{1}{k!} ~l_k (\underbrace{\alpha, \ldots, \alpha}_{k \text{ times}}) = 0.
\end{align*}
\end{defn}

Let $( \mathfrak{L}, \{ l_k \}_{k \geq 1}  )$ be an $L_\infty$-algebra and $\alpha$ be a Maurer-Cartan element of it. Then using $\alpha$, one can construct a new $L_\infty$-algebra on the same graded vector space $\mathfrak{L}$ with the structure maps $\{ l_k^\alpha \}_{k \geq 1}$ given by
\begin{align*}
l_k^\alpha (x_1, \ldots, x_k) := \sum_{i=0}^\infty \frac{1}{i!} ~l_{k+i} (\underbrace{\alpha, \ldots, \alpha}_{i \text{ times}}, x_1, \ldots, x_k), \text{ for } x_1, \ldots, x_k \in \mathfrak{L}.
\end{align*}
The $L_\infty$-algebra $(\mathfrak{L}, \{ l_k^\alpha \}_{k \geq 1})$ is said to be obtained from $(\mathfrak{L}, \{ l_k \}_{k \geq 1})$ twisting by the Maurer-Cartan element $\alpha$.

$L_\infty$-algebras often arise from graded Lie algebras endowed with certain additional data. The structure maps for such $L_\infty$-algebras are obtained via derived brackets.

\begin{defn}
A {\bf $V$-data} is a quadruple $(\mathfrak{g}, \mathfrak{a}, P, \triangle)$ consisting of a graded Lie algebra $\mathfrak{g}$ (with the graded Lie bracket $[~,~]$), an abelian Lie subalgebra $\mathfrak{a} \subset \mathfrak{g}$, a projection map $P: \mathfrak{g} \rightarrow \mathfrak{g}$ with $\mathrm{im}(P) = \mathfrak{a}$ and $\mathrm{ker}(P) \subset \mathfrak{g}$ a graded Lie subalgebra, and $\triangle \in \mathrm{ker}(P)_1$ that satisfies $[\triangle, \triangle] = 0$.
\end{defn}

\begin{thm}\label{voro-thm}
Let $(\mathfrak{g}, \mathfrak{a}, P, \triangle)$ be a $V$-data. Then the graded vector space $s^{-1} \mathfrak{g} \oplus \mathfrak{a}$ inherits an $L_\infty$-algebra with the structure maps $\{ l_k \}_{k \geq 1}$ are given by
\begin{align}
l_1 \big( (s^{-1}x,a)  \big)  =~& \big( - s^{-1} [\triangle, x], ~ P (x + [\triangle, a])  \big), \\
l_2 \big( (s^{-1}x, 0), (s^{-1}y, 0)  \big) =~&  \big( (-1)^{|x|} s^{-1} [x,y],~ 0  \big),  \\
l_k \big( (s^{-1}x, 0), (0, a_1), \ldots, (0, a_{k-1})  \big) =~&  \big( 0, ~ P[ \cdots [ ~[ x, a_1], a_2], \ldots, a_{k-1}]  \big), \text{ for } k \geq 2,  \\
l_k \big( (0, a_1), \ldots, (0, a_{k})  \big) =~&   \big( 0, ~  ~ P[ \cdots [ ~[ \triangle , a_1], a_2], \ldots, a_{k}] \big), \text{ for } k \geq 1, \label{a-bracket} \\
\text{ and up to permutations of the} & \text{ above inputs, all other maps vanish}. \nonumber
\end{align}
Here $x, y$ are homogeneous elements of $\mathfrak{g}$ and $a, a_1, \ldots, a_k$ are homogeneous elements of $\mathfrak{a}$.
\end{thm}

The above result has the following consequences.
\begin{thm}\label{cons-thm}
(a) Let $(\mathfrak{g}, \mathfrak{a}, P, \triangle)$ be a $V$-data. Then $(\mathfrak{a}, \{ l_k \}_{k \geq 1})$ is an $L_\infty$-algebra, where the structure maps $\{ l_k \}_{k \geq 1}$ are given by (\ref{a-bracket}).

(b) Let $(\mathfrak{g}, \mathfrak{a}, P, \triangle)$ be a $V$-data and $\mathfrak{h} \subset \mathfrak{g}$ be a graded Lie subalgebra that satisfies $[\triangle, \mathfrak{h}] \subset \mathfrak{h}$. Then $(s^{-1} \mathfrak{h} \oplus \mathfrak{a}, \{ l_k \}_{k \geq 1})$ is an $L_\infty$-algebra, where the structure maps are given by the formulas given in Theorem \ref{voro-thm}.
\end{thm}

\medskip

\section{Maurer-Cartan characterization of relative Rota-Baxter family algebras}\label{sec3} In this section, our main aim is to construct an $L_\infty$-algebra that characterizes relative Rota-Baxter family algebras as its Maurer-Cartan elements. Along the way, we will prove some other relevant results.

\subsection{A generalization of the Gerstenhaber's graded Lie algebra} Let $V$ be any vector space. For any natural number $n \in \mathbb{N}$, let $\mathrm{Hom}_\Omega (V^{\otimes n}, V)$ be the set whose elements are given by a collection $f = \{ f_{\alpha_1, \ldots, \alpha_n} : V^{\otimes n} \rightarrow V \}_{ \alpha_1, \ldots, \alpha_n \in \Omega }$ of linear maps labelled by the elements of $\Omega^{\times n}$. Note that the set $\mathrm{Hom}_\Omega (V^{\otimes n}, V)$ has an obvious vector space structure with the addition and scalar multiplication respectively given by
\begin{align*}
f + f' = \big\{ f_{\alpha_1, \ldots, \alpha_n} + f'_{\alpha_1, \ldots, \alpha_n} \big\}_{\alpha_1, \ldots, \alpha_n \in \Omega} ~~~~ \text{ and } ~~~~ \lambda f  = \big\{ \lambda f_{\alpha_1, \ldots, \alpha_n} \big\}_{\alpha_1, \ldots, \alpha_n \in \Omega},
\end{align*}
for $f = \{ f_{\alpha_1, \ldots, \alpha_n} \}_{ \alpha_1, \ldots, \alpha_n \in \Omega }$, $f' = \{ f'_{\alpha_1, \ldots, \alpha_n} \}_{ \alpha_1, \ldots, \alpha_n \in \Omega } \in \mathrm{Hom}_\Omega (V^{\otimes n}, V)$ and $\lambda \in {\bf k}$. For any $m,n \geq 1$ and $1 \leq i \leq m$, there is a linear map $\circ_i : \mathrm{Hom}_\Omega (V^{\otimes m}, V) \otimes \mathrm{Hom}_\Omega (V^{\otimes n}, V) \rightarrow \mathrm{Hom}_\Omega (V^{\otimes m+n-1}, V)$, $(f, g) \mapsto f \circ_i g~$ given by
\begin{align}\label{partial-maps}
(f ~ \circ_i &~g)_{\alpha_1, \ldots, \alpha_{m+n-1}} (v_1, \ldots, v_{m+n-1}) \\ &:= f_{\alpha_1, \alpha_{i-1}, \alpha_i \cdots \alpha_{i+n-1}, \ldots, \alpha_{m+n-1}} \big( v_1, \ldots, v_{i-1}, g_{\alpha_i, \ldots, \alpha_{i+n-1}} (v_i, \ldots, v_{i+n-1}), v_{i+n}, \ldots, v_{m+n-1}  \big), \nonumber
\end{align}
for $\alpha_1, \ldots, \alpha_{m+n-1} \in \Omega$ and $v_1, \ldots, v_{m+n-1} \in V$. Such maps (which are called {\em partial compositions}) satisfy the followings
\begin{align*}
(f \circ_i g) \circ_{i+j-1} h =~& f \circ_i (g \circ_j h), \text{ for } 1 \leq i \leq m,~ 1 \leq j \leq n, \\
(f \circ_i g) \circ_{j+n-1} h =~& (f \circ_j h) \circ_i g, \text{ for } 1 \leq i < j \leq m,
\end{align*}
for $f \in \mathrm{Hom}_\Omega (V^{\otimes m}, V)$, $g \in \mathrm{Hom}_\Omega (V^{\otimes n}, V)$ and $h \in \mathrm{Hom}_\Omega (V^{\otimes p}, V)$.
Moreover, there is an element $I = \{ I_\alpha : V \rightarrow V \}_{\alpha \in \Omega} \in \mathrm{Hom}_\Omega (V, V)$ given by $I_\alpha = \mathrm{id}_V$, for all $\alpha \in \Omega$. The element $I \in \mathrm{Hom}_\Omega (V, V)$ satisfies
\begin{align*}
f \circ_i I = f = I \circ_1 f,
\end{align*}
for any $f \in \mathrm{Hom}_\Omega (V^{\otimes m}, V)$ and $1 \leq i \leq m$. Thus, the collection of vector spaces $\{ \mathrm{Hom}_\Omega (V^{\otimes n}, V) \}_{n \geq 1}$ equipped with the partial compositions (\ref{partial-maps}) forms a nonsymmetric operad. As a consequence, the graded vector space $\oplus_{n \geq 1} \mathrm{Hom}_\Omega (V^{\otimes n}, V)$ carries a degree $-1$ graded Lie bracket (called the {\bf $\Omega$-Gerstenhaber bracket}) given by
\begin{align}
[f, g]_\Omega := \sum_{i=1}^m (-1)^{(i-1)(n-1)}~ f \circ_i g - (-1)^{(m-1)(n-1)} \sum_{i=1}^n (-1)^{(i-1) (m-1)}~ g \circ_i f,
\end{align}
for $f \in \mathrm{Hom}_\Omega (V^{\otimes m}, V)$ and $g \in \mathrm{Hom}_\Omega (V^{\otimes n}, V)$. In other words, $\big( \oplus_{n \geq 0} \mathrm{Hom}_\Omega (V^{\otimes n+1}, V), [~,~]_\Omega     \big)$ is a graded Lie algebra.

\medskip

\subsection{A new graded Lie algebra}\label{subsection-a graded lie} Given an associative algebra $A$ and an $A$-bimodule $M$, here we construct a new graded Lie algebra using the $\Omega$-Gerstenhaber bracket. A Maurer-Cartan element of this graded Lie algebra corresponds to a collection $\{ R_\alpha : M \rightarrow A \}_{\alpha \in \Omega}$ of linear maps satisfying (\ref{rel-rb-fam-iden}). 

Let $A$ be an associative algebra and $M$ be an $A$-bimodule. 
Let $\mu : A \otimes A \rightarrow A$ denotes the associative multiplication on $A$, i.e. $\mu (a , b) = a \cdot b$, for $a , b \in A$. Moreover, we denote the left and right $A$-actions on $M$ by the maps $l$ and $r$, respectively. Consider the graded Lie algebra 
\begin{align*}
\mathfrak{g} = \big( \oplus_{n \geq 0} \mathrm{Hom}_\Omega ( (A \oplus M)^{\otimes n+1}, A \oplus M), [~,~]_\Omega     \big)
\end{align*}
associated with the direct sum vector space $A \oplus M$. Then we have the followings.

\medskip

(i) The maps $\mu, l, r$ induces an element $\triangle \in  \mathrm{Hom}_\Omega ( (A \oplus M)^{\otimes 2}, A \oplus M)$ given by
\begin{align*}
\triangle_{\alpha, \beta} \big(  (a, u), (b, v)  \big) := \big(  \mu (a, b), ~ l(a, v) + r (u, b)  \big) = (a \cdot b ,~ a \cdot_M v + u \cdot_M b),
\end{align*}
for all $\alpha, \beta \in \Omega$ and $(a, u), (b, v ) \in A \oplus M$. The element $\triangle \in \mathrm{Hom}_\Omega ( (A \oplus M)^{\otimes 2}, A \oplus M)$ satisfies $[\triangle, \triangle]_\Omega = 0$.

\medskip

(ii) The graded subspace $\mathfrak{a} = \oplus_{n \geq 0} \mathrm{Hom}_\Omega ( M^{\otimes n+1}, A) \subset \mathfrak{g}$ is an abelian Lie subalgebra.

\medskip

Let $P: \mathfrak{g} \rightarrow \mathfrak{g}$ be the projection onto the subspace $\mathfrak{a}$. Then it follows that $(\mathfrak{g}, \mathfrak{a}, P, \triangle)$ is a $V$-data. As a consequence, the graded vector space $\mathfrak{a}$ inherits an $L_\infty$-algebra structure (see Theorem \ref{cons-thm} (a)). It turns out that all the structure maps of this $L_\infty$-algebra vanish except the map $l_2$. Therefore, the suspension $s \mathfrak{a} = \oplus_{n \geq 1} \mathrm{Hom}_\Omega (M^{\otimes n}, A)$ carries a graded Lie algebra structure with the bracket
\begin{align}\label{deri-bracket}
\llbracket f, g \rrbracket := (-1)^{m} ~ [ [\triangle, f]_\Omega, g ]_\Omega,
\end{align}
for $f  \in \mathrm{Hom}_\Omega (M^{\otimes m}, A)$ and $g \in \mathrm{Hom}_\Omega (M^{\otimes n}, A)$. As a collection, the bracket (\ref{deri-bracket}) is given by
\begin{align}\label{deri-bracker-explicit}
&\llbracket f, g \rrbracket_{\alpha_1, \ldots, \alpha_{m+n}} (u_1, \ldots, u_{m+n}) \\
&= \sum_{i=1}^m (-1)^{(i-1)n} f_{\alpha_1, \ldots, \alpha_i \cdots \alpha_{i+n}, \ldots, \alpha_{m+n}} \big(  u_1, \ldots, u_{i-1}, g_{\alpha_i, \ldots, \alpha_{i+n-1}} (u_i, \ldots, u_{i+n-1}) \cdot_M u_{i+n}, \ldots, u_{m+n} \big)  \nonumber\\
& ~~ - \sum_{i=1}^m (-1)^{in} f_{\alpha_1, \ldots, \alpha_i \cdots \alpha_{i+n}, \ldots, \alpha_{m+n}} \big(  u_1, \ldots, u_{i-1}, u_i \cdot_M g_{\alpha_{i+1}, \ldots, \alpha_{i+n}} (u_{i+1}, \ldots, u_{i+n}), \ldots, u_{m+n} \big)  \nonumber\\
& ~~ - (-1)^{mn} \bigg\{  \sum_{i=1}^n (-1)^{(i-1)m} g_{\alpha_1, \ldots, \alpha_i \cdots \alpha_{i+m}, \ldots, \alpha_{m+n}} \big(  u_1, \ldots, f_{\alpha_i, \ldots, \alpha_{i+m-1}} (u_i, \ldots, u_{i+m-1}) \cdot_M u_{i+m}, \ldots, u_{m+n} \big)  \nonumber\\
& ~~ - \sum_{i=1}^n (-1)^{im}~ g_{\alpha_1, \ldots, \alpha_i \cdots \alpha_{i+m}, \ldots, \alpha_{m+n}} \big(  u_1, \ldots, u_{i-1}, u_i \cdot_M f_{\alpha_{i+1}, \ldots, \alpha_{i+m}} (u_{i+1}, \ldots, u_{i+m}), \ldots, u_{m+n} \big)      \bigg\}  \nonumber\\
& ~~ + (-1)^{mn} \big\{ f_{\alpha_1, \ldots, \alpha_m} (u_1, \ldots, u_m) \cdot g_{\alpha_{m+1}, \ldots, \alpha_{m+n}} (u_{m+1}, \ldots, u_{m+n})  \nonumber\\
& ~~ \qquad \qquad \qquad  - (-1)^{mn} g_{\alpha_1, \ldots, \alpha_n} (u_1, \ldots, u_n) \cdot f_{\alpha_{n+1}, \ldots, \alpha_{m+n}} (u_{n+1}, \ldots, u_{m+n})  \big\}, \nonumber
\end{align}
for $f = \{ f_{\alpha_1, \ldots, \alpha_m} \}_{\alpha_1, \ldots, \alpha_m \in \Omega}$, $g = \{ g_{\alpha_1, \ldots, \alpha_n} \}_{\alpha_1, \ldots, \alpha_n \in \Omega}$ and $u_1, \ldots, u_{m+n} \in M$.
This graded Lie bracket can be extended to the graded vector space $\oplus_{n \geq 0} \mathrm{Hom}_\Omega (M^{\otimes n}, A)$ by assuming that $\mathrm{Hom}_\Omega (M^{\otimes 0}, A) = A$. The formulas for the extension are given by
\begin{align*}
\llbracket f, a \rrbracket_{\alpha_1, \ldots, \alpha_m} (u_1, \ldots, u_m) =~& \sum_{i=1}^m f_{\alpha_1, \ldots, \alpha_m} \big(  u_1, \ldots, u_{i-1}, a \cdot_M u_i - u_i \cdot_M a, u_{i+1}, \ldots, u_m  \big) \\
&+ f_{\alpha_1, \ldots, \alpha_m} (u_1, \ldots, u_m ) \cdot a - a \cdot f_{\alpha_1, \ldots, \alpha_m} (u_1, \ldots, u_m ), \\
\llbracket a, b \rrbracket =~& a \cdot b - b \cdot a,
\end{align*}
for $f \in \mathrm{Hom}_\Omega (M^{\otimes m}, A)$ and $a, b \in A = \mathrm{Hom}_\Omega (M^{\otimes 0}, A)$. Thus, we obtain the following result.

\begin{thm}\label{mc-gla-thm}
Let $A$ be an associative algebra and $M$ be an $A$-bimodule. Then $ \big(  \oplus_{n \geq 0} \mathrm{Hom}_\Omega (M^{\otimes n}, A), \llbracket ~, ~ \rrbracket \big)$ is a graded Lie algebra. A collection $R = \{ R_\alpha : M \rightarrow A \}_{\alpha \in \Omega}$ of linear maps satisfy the identities (\ref{rel-rb-fam-iden}) if and only if $R = \{ R_\alpha \}_{\alpha \in \Omega} \in \mathrm{Hom}_\Omega (M, A)$ is a Maurer-Cartan element in the graded Lie algebra $\big(  \oplus_{n \geq 0} \mathrm{Hom}_\Omega (M^{\otimes n}, A), \llbracket ~, ~ \rrbracket \big)$.
\end{thm}

\begin{proof}
For any $R = \{ R_\alpha \}_{\alpha \in \Omega} \in \mathrm{Hom}_\Omega (M,A)$, we have from (\ref{deri-bracker-explicit}) that
\begin{align*}
\llbracket R, R \rrbracket_{\alpha , \beta} (u, v) = 2 \big(  R_\alpha (u) \cdot R_\beta (v) - R_{\alpha \beta} (R_\alpha (u) \cdot_M v ) - R_{\alpha \beta} (u \cdot_M R_\beta (v) ) \big),
\end{align*}
for $\alpha, \beta \in \Omega$ and $u, v \in M$. This shows that $\llbracket R, R \rrbracket = 0$ if and only if the collection $\{ R_\alpha \}_{\alpha \in \Omega}$ satisfy (\ref{rel-rb-fam-iden}). This completes the proof.
\end{proof}

Let $(A, M, \{ R_\alpha \}_{\alpha \in \Omega})$ be a relative Rota-Baxter family algebra. Then it follows from the previous theorem that $R = \{ R_\alpha \}_{\alpha \in \Omega}$ is a Maurer-Cartan element in the graded Lie algebra $\big(  \oplus_{n \geq 0} \mathrm{Hom}_\Omega (M^{\otimes n}, A), \llbracket ~, ~ \rrbracket \big)$. Hence it induces a differential
\begin{align*}
d_R : \mathrm{Hom}_\Omega (M^{\otimes n}, A) \rightarrow \mathrm{Hom}_\Omega (M^{\otimes n+1}, A), ~ d_R (f) = (-1)^n \llbracket R, f \rrbracket, \text{ for } f \in \mathrm{Hom}_\Omega (M^{\otimes n}, A).
\end{align*}
Explicitly, the differential $d_R$ is given by
\begin{align*}
&(d_R (f))_{\alpha_1, \ldots, \alpha_{n+1}} (u_1, \ldots, u_{n+1}) \\
&= R_{\alpha_1} (u_1) \cdot f_{\alpha_2, \ldots, \alpha_{n+1}} (u_2, \ldots, u_{n+1}) - R_{\alpha_1 \cdots \alpha_{n+1}} \big(  u_1 \cdot_M f_{\alpha_2, \ldots, \alpha_{n+1}} (u_2, \ldots, u_{n+1}) \big) \\
& ~~ + \sum_{i=1}^n (-1)^i ~ f_{\alpha_1, \ldots, \alpha_i \alpha_{i+1}, \ldots, \alpha_{n+1}} \big(  u_1, \ldots, u_{i-1}, R_{\alpha_i} (u_i) \cdot_M u_{i+1} + u_i \cdot_M R_{\alpha_{i+1}} (u_{i+1}), \ldots, u_{n+1}  \big) \\
& ~~ + (-1)^{n+1} f_{\alpha_1, \ldots, \alpha_n} (u_1, \ldots, u_n) \cdot R_{\alpha_{n+1}} (u_{n+1}) - (-1)^{n+1} R_{\alpha_1 \cdots \alpha_{n+1}} \big(  f_{\alpha_1, \ldots, \alpha_{n}} (u_1, \ldots, u_n) \cdot_M u_{n+1}  \big),
\end{align*}
for $\alpha_1, \ldots, \alpha_{n+1} \in \Omega$ and $u_1, \ldots, u_{n+1} \in M$. In other words, $\big\{ \mathrm{Hom}_\Omega (M^{\otimes \bullet}, A)  , d_R \big\}$ is a cochain complex. The corresponding cohomology groups are called the {\bf cohomology} induced by the collection $R = \{ R_\alpha \}_{\alpha \in \Omega}$, and they are denoted by $H^\bullet_R (M,A)$. 

\medskip

\subsection{A new $L_\infty$-algebra} In this subsection, we construct an $L_\infty$-algebra that characterizes relative Rota-Baxter family algebras as its Maurer-Cartan elements. We will use this characterization in the next section to define the cohomology of a given relative Rota-Baxter family algebra.

Let $A$ and $M$ be two vector spaces (not necessarily equipped with any additional structures). For any $k, l \geq 0$, let $\mathcal{A}^{k, l}$ be the direct sum of all possible $(k+l)$ tensor powers of $A$ and $M$ in which $A$ appears $k$ times and $M$ appears $l$ times. For example, we have
\begin{align*}
\mathcal{A}^{1,1} = (A \otimes M) \oplus (M \otimes A), \quad \mathcal{A}^{2,0} = A \otimes A ~~~~ \text{ and } ~~~~ \mathcal{A}^{0,2} = M \otimes M.
\end{align*}
Next, consider the graded Lie algebra $\mathfrak{g} = \big(   \oplus_{n \geq 0} \mathrm{Hom}_\Omega ((A \oplus M)^{\otimes n+1}, A \oplus M), [~,~]_\Omega \big)$
associated to the vector space $A \oplus M$. Note that the space $\mathrm{Hom}_\Omega ((A \oplus M)^{\otimes n+1}, A \oplus M)$ has a decomposition
\begin{align*}
\mathrm{Hom}_\Omega ((A \oplus M)^{\otimes n+1}, A \oplus M) \cong  \big(  \oplus_{k+l = n+1} \mathrm{Hom}_\Omega (\mathcal{A}^{k,l}, A)  \big) \oplus \big(   \oplus_{k+l = n+1} \mathrm{Hom}_\Omega (\mathcal{A}^{k,l}, M)   \big).
\end{align*}
An element $f \in \mathrm{Hom}_\Omega ((A \oplus M)^{\otimes n+1}, A \oplus M) $ is said to have {\bf bidegree} $k| l$ with $k+l = n$ if
\begin{align*}
f \in \mathrm{Hom}_\Omega (\mathcal{A}^{k+1,l}, A) \oplus \mathrm{Hom}_\Omega (\mathcal{A}^{k, l+1}, M) \subset \mathrm{Hom}_\Omega ((A \oplus M)^{\otimes n+1}, A \oplus M).
\end{align*}
We denote the set of all elements of $\mathrm{Hom}_\Omega ((A \oplus M)^{\otimes n+1}, A \oplus M)$ of bidegree $k|l$ with $k+l= n$ by $\mathrm{Hom}^{k|l}_\Omega ( (A \oplus M)^{\otimes n+1}, A \oplus M)$. Then we have
\begin{align*}
\mathrm{Hom}^{k|0}_\Omega ( (A \oplus M)^{\otimes k+1}, A \oplus M) \cong ~& \mathrm{Hom}_\Omega (A^{\otimes k+1}, A) \oplus \mathrm{Hom}_\Omega (\mathcal{A}^{k,1}, M), \\
\mathrm{Hom}^{-1|l}_\Omega ( (A \oplus M)^{\otimes l}, A \oplus M) \cong ~& \mathrm{Hom}_\Omega (M^{\otimes l}, A).
\end{align*}
It is not hard to see that the $\Omega$-Gerstenhaber bracket $[~,~]_\Omega$ satisfies the following.

\begin{prop}
For $f \in \mathrm{Hom}^{k|l}_\Omega ( (A \oplus M)^{\otimes k+l +1}, A \oplus M)$ and $g \in \mathrm{Hom}^{p|q}_\Omega ( (A \oplus M)^{\otimes p+q+1}, A \oplus M)$, we have
\begin{align*}
[f,g]_\Omega \in \mathrm{Hom}^{k+p| l+q}_\Omega ( (A \oplus M)^{\otimes k+l +p+q+1}, A \oplus M).
\end{align*}
\end{prop}

It follows from the previous proposition that $\mathfrak{a} = \oplus_{l \geq 0} \mathrm{Hom}^{-1 | l+1}_\Omega ( (A \oplus M)^{\otimes l+1}, A \oplus M)$ is an abelian Lie subalgebra of $\mathfrak{g}$. Let $P: \mathfrak{g} \rightarrow \mathfrak{g}$ be the projection onto the subspace $\mathfrak{a}$. Then the quadruple $(\mathfrak{g}, \mathfrak{a}, P, \triangle = 0)$ is a $V$-data.

The above proposition also suggests that the graded subspace 
\begin{align*}
\mathfrak{h} = \oplus_{n \geq 0}~ \mathrm{Hom}^{n|0}_\Omega ((A \oplus M)^{\otimes n+1}, A \oplus M)
\end{align*}
 is a graded Lie subalgebra of $\mathfrak{g}$. Hence by Theorem \ref{cons-thm} (b), the graded space 
$ s^{-1} \mathfrak{h} \oplus \mathfrak{a} = \oplus_{n \geq -1} (s^{-1} \mathfrak{h} \oplus \mathfrak{a})_n$ carries an $L_\infty$-algebra structure. Here the graded space $ s^{-1} \mathfrak{h} \oplus \mathfrak{a}$ is given by
\begin{align*}
    (s^{-1} \mathfrak{h} \oplus \mathfrak{a})_{-1} =~&  \mathfrak{h}_0 = \mathrm{Hom}_\Omega (A, A) \oplus \mathrm{Hom}_\Omega (M,M) ~~~ \quad (\text{we assume that } \mathrm{a}_{-1} = 0),\\
    (s^{-1} \mathfrak{h} \oplus \mathfrak{a})_n =~& \mathfrak{h}_{n+1} \oplus \mathfrak{a}_n = \mathrm{Hom}_\Omega (A^{\otimes n+2}, A) \oplus \mathrm{Hom}_\Omega (\mathcal{A}^{n+1, 1}, M) \oplus \mathrm{Hom}_\Omega (M^{\otimes n+1}, A), \text{ for } n \geq 0.
\end{align*}
The operations $\{ l_k \}_{k \geq 1}$ for the $L_\infty$-structure are given by
\begin{align*}
l_2 (s^{-1} q, s^{-1} q') =~& (-1)^{|q|} s^{-1} [q,q']_\Omega, \\
l_k (s^{-1} q, p_1, \ldots, p_{k-1}) =~& P [ \cdots [[ q, p_1]_\Omega, p_2 ]_\Omega, \ldots, p_{k-1} ]_\Omega, \text{ for } k \geq 2,
\end{align*}
for homogeneous elements $q, q' \in \mathfrak{h}$ (considered as elements $s^{-1}q,~ s^{-1} q' \in s^{-1} \mathfrak{h}$) and homogeneous elements $p_1, \ldots, p_{k-1} \in \mathfrak{a}$. Upto permutations of the above inputs, all other maps vanish.

\medskip

Note that there is an embedding
\begin{align}\label{graded-embedd}
   \oplus_{n \geq 0} \mathrm{Hom} \big(  (A \oplus M)^{\otimes n+1}, A \oplus M  \big) ~ \hookrightarrow ~ \oplus_{n \geq 0} \mathrm{Hom}_\Omega \big(  (A \oplus M)^{\otimes n+1}, A \oplus M  \big), ~ f \mapsto \widetilde{f}
\end{align}
given by
\begin{align*}
    (\widetilde{f})_{\alpha_1, \ldots, \alpha_{n+1}} \big(  (a_1, u_1), \ldots, (a_{n+1}, u_{n+1}) \big) = f \big(  (a_1, u_1), \ldots, (a_{n+1}, u_{n+1}) \big),
\end{align*}
for all $\alpha_1, \ldots, \alpha_{n+1} \in \Omega$ and $(a_1, u_1), \ldots, (a_{n+1}, u_{n+1}) \in A \oplus M$. This restricts to an embedding
\begin{align*}
  \oplus_{n \geq 0}  \mathrm{Hom}^{n | 0} \big( (A \oplus M)^{\otimes n+1}, A \oplus M  \big) \hookrightarrow  \oplus_{n \geq 0} \mathrm{Hom}^{n | 0}_\Omega \big( (A \oplus M)^{\otimes n+1}, A \oplus M  \big) , ~ f \mapsto \widetilde{f}.
\end{align*}
which yields an embedding 
\begin{align*}
\mathfrak{h}' = \oplus_{n \geq 0} \mathrm{Hom}^{n|0} \big( (A \oplus M)^{\otimes n+1}, A \oplus M  \big) = \oplus_{n \geq 0} \big( \mathrm{Hom}(A^{\otimes n+1}, A) \oplus \mathrm{Hom}(\mathcal{A}^{n,1}, M) \big) ~ \hookrightarrow \mathfrak{h}
\end{align*}
at the level of graded vector spaces. With this embedding, the $L_\infty$-algebra structure on $s^{-1} \mathfrak{h} \oplus \mathfrak{a}$ induces an $L_\infty$-algebra structure on the graded vector space $s^{-1} \mathfrak{h}' \oplus \mathfrak{a}$. More precisely, we have the following.

\begin{thm}
Let $A$ and $M$ be two vector spaces. Then the graded vector space $s^{-1} \mathfrak{h}' \oplus \mathfrak{a}$
carries an $L_\infty$-algebra structure with the operations $\{ l_k \}_{k \geq 1}$ are given by
\begin{align*}
l_2 (s^{-1} q, s^{-1} q') =~& (-1)^{|q|} s^{-1} [ \widetilde{q}, \widetilde{q'}]_\Omega, \\
l_k (s^{-1} q, p_1, \ldots, p_{k-1}) =~& P [ \cdots [[ \widetilde{q}, p_1]_\Omega, p_2 ]_\Omega, \ldots, p_{k-1} ]_\Omega, \text{ for } k \geq 2.
\end{align*}
\end{thm}

\medskip

The importance of the above $L_\infty$-algebra is given by the following Maurer-Cartan characterization of relative Rota-Baxter family algebras.

\begin{thm}\label{mc-main-thm}
Let $A$ and $M$ be two vector spaces. Suppose there are linear maps
\begin{align*}
\mu : A \otimes A \rightarrow A, \quad l : A \otimes M \rightarrow M, \quad r : M \otimes A \rightarrow M
\end{align*}
and a collection of linear maps $R = \{ R_\alpha : M \rightarrow A \}_{\alpha \in \Omega}$ labelled by the elements of the semigroup $\Omega$. Then the triple $\big( A=  (A, \mu),  M= (M,l,r), \{ R_\alpha \}_{\alpha \in \Omega} \big)$ is a relative Rota-Baxter family algebra if and only if $\alpha = (s^{-1} \triangle, R) \in (s^{-1} \mathfrak{h}' \oplus \mathfrak{a})_0$ is a Maurer-Cartan element in the $L_\infty$-algebra $(s^{-1} \mathfrak{h}' \oplus \mathfrak{a}, \{ l_k \}_{k \geq 1}),$ where $\triangle = \mu + l + r.$
\end{thm}

\begin{proof}
    Note that $l_1 \big( (s^{-1} \triangle, R) \big) = 0$ (from the definition of $l_1$). Moreover, from the homogeneous degree reason, we have
    \begin{align*}
        [[[ \widetilde{\triangle}, R]_\Omega , R]_\Omega, R]_\Omega = 0.
    \end{align*}
    Therefore, $l_k \big( (s^{-1} \triangle, R), \ldots, (s^{-1} \triangle, R)   \big) = 0$ for $k \geq 4$. Hence we have
    \begin{align*}
       & \sum_{k=1}^\infty \frac{1}{k!} ~ l_k \big(    (s^{-1} \triangle, R), \ldots, (s^{-1} \triangle, R) \big) \\
       & = \frac{1}{2!}~ l_2 \big(  (s^{-1} \triangle, R), (s^{-1} \triangle, R)  \big) + \frac{1}{3!} ~ l_3 \big(  (s^{-1} \triangle, R), (s^{-1} \triangle, R), (s^{-1} \triangle, R) \big) \\
       & = \big(  - \frac{1}{2} s^{-1} [\widetilde{\triangle}, \widetilde{\triangle} ]_\Omega , ~ \frac{1}{2} [[ \widetilde{\triangle}, R]_\Omega, R]_\Omega  \big).
    \end{align*}
    This shows that $\alpha = (s^{-1} \triangle, R)$ is a Maurer-Cartan element in the $L_\infty$-algebra $(s^{-1} \mathfrak{h}' \oplus \mathfrak{a}, \{ l_k \}_{k \geq 1})$ if and only if $[\widetilde{\triangle}, \widetilde{\triangle} ]_\Omega = 0$ and $ [[ \widetilde{\triangle}, R]_\Omega, R]_\Omega   = 0$.

    Note that the $\Omega$-Gerstenhaber bracket $[~,~]_\Omega$ coincides with the usual Gerstenhaber bracket (denoted by $[~,~]_\mathsf{G}$) on the embedded subspace given in (\ref{graded-embedd}). Hence
    \begin{align*}
         [\widetilde{\triangle}, \widetilde{\triangle}]_\Omega = 0 ~ \Longleftrightarrow ~ [\triangle, \triangle]_\mathsf{G} = 0.
    \end{align*}
This holds if and only if $(A, \mu)$ is an associative algebra and $(M,l,r)$  is an $(A, \mu)$-bimodule (see \cite[Proposition 2.11]{das-rota}).

    On the other hand, $[[ \widetilde{\triangle}, R]_\Omega, R]_\Omega   = 0$ holds if and only if $\llbracket R, R \rrbracket = 0$. By Theorem \ref{mc-gla-thm}, this condition is equivalent to the fact that the collection $R = \{ R_\alpha \}_{\alpha \in \Omega}$ satisfy the identities  (\ref{rel-rb-fam-iden}). Combining these, we get that $\alpha = (s^{-1} \triangle, R)$ is a Maurer-Cartan element if and only if $\big( (A, \mu), (M,l, r), \{ R_\alpha\}_{\alpha \in \Omega}  \big)$ is a relative Rota-Baxter family algebra.
\end{proof}

\section{Cohomology of relative Rota-Baxter family algebras}\label{sec4}
In this section, we define the cohomology of relative Rota-Baxter family algebras using their  Maurer-Cartan characterizations. In particular, we define the cohomology of Rota-Baxter family algebras. Finally, we obtain a long exact sequence that connects various cohomology groups.

\medskip

\subsection{Cohomology of relative Rota-Baxter family algebras}
Let $(A, M, \{ R_\alpha \}_{\alpha \in \Omega})$ be a relative Rota-Baxter family algebra. Then we have seen in Theorem \ref{mc-main-thm} that $\alpha = (s^{-1} \triangle, R)$ is a Maurer-Cartan element in the $L_\infty$-algebra $(s^{-1} \mathfrak{h}' \oplus \mathfrak{a}, \{ l_k \}_{k \geq 1})$, where $\triangle = \mu + l + r$. Hence we can consider the $L_\infty$-algebra $(s^{-1} \mathfrak{h}' \oplus \mathfrak{a}, \{ l_k^{(s^{-1} \triangle, R) } \}_{k \geq 1})$ twisted by the Maurer-Cartan element $\alpha = (s^{-1} \triangle, R)$. Thus, it follows that the degree $1$ map $l_1^{(s^{-1} \triangle, R)} : s^{-1} \mathfrak{h}' \oplus \mathfrak{a} \rightarrow s^{-1} \mathfrak{h}' \oplus \mathfrak{a}$ satisfies $( l_1^{(s^{-1} \triangle, R)}  )^2 = 0$. We will use this observation to define the cochain complex associated with the relative Rota-Baxter family algebra.

For each $n \geq 0$, we denote the space of $n$-cochains by $C^n_\mathrm{rRBf} (A,M, \{ R_\alpha \}_{\alpha \in \Omega})$ and is defined by
\begin{align*}
    C^n_\mathrm{rRBf} (A,M, \{ R_\alpha \}_{\alpha \in \Omega}) = \begin{cases}
        0 & \text{ if } n =0, \\
        \mathrm{Hom}(A,A) \oplus \mathrm{Hom}(M,M) & \text{ if } n =1,\\
        \mathrm{Hom}(A^{\otimes n}, A) \oplus \mathrm{Hom}(\mathcal{A}^{n-1,1}, M) \oplus \mathrm{Hom}_\Omega (M^{\otimes n-1}, A) & \text{ if } n \geq 2.
    \end{cases}
\end{align*}
An element $(f, g) \in C^1_\mathrm{rRBf} (A, M, \{ R_\alpha \}_{\alpha \in \Omega})$ gives rise to an element $(s^{-1} (f + g), 0) \in (s^{-1} \mathfrak{h}' \oplus \mathfrak{a})_{{-1}}$. On the other hand, for any $n \geq 2$, the cochain group $C^n_\mathrm{rRBf} (A, M, \{ R_\alpha \}_{\alpha \in \Omega})$ can be identified with $(s^{-1} \mathfrak{h}' \oplus \mathfrak{a})_{n-2}$ via
\begin{align*}
\mathrm{Hom} (A^{\otimes n}, A) \oplus \mathrm{Hom} (\mathcal{A}^{n-1,1}, M) \oplus \mathrm{Hom}_\Omega (M^{\otimes n-1}, A) \ni (f,g, \gamma) ~ \leftrightsquigarrow ~ (s^{-1} (f + g), \gamma) \in (s^{-1} \mathfrak{h}' \oplus \mathfrak{a})_{n-2}.
\end{align*}
Using these identifications, we define a map $\delta_\mathrm{rRBf} : C^n_\mathrm{rRBf} (A, M, \{ R_\alpha \}_{\alpha \in \Omega}) \rightarrow C^{n+1}_\mathrm{rRBf} (A, M, \{ R_\alpha \}_{\alpha \in \Omega})$ by
\begin{align*}
\delta_\mathrm{rRBf} \big( (f,g) \big) =~& - l_1^{(s^{-1} \triangle, R)} \big( s^{-1} (f+g), 0   \big), \text{ for } (f,g) \in C^1_\mathrm{rRBf} (A, M, \{ R_\alpha \}_{\alpha \in \Omega}),\\
\delta_\mathrm{rRBf} \big( (f, g, \gamma) \big) =~& (-1)^{n-2}~ l_1^{(s^{-1} \triangle, R)} \big( s^{-1} (f + g), \gamma   \big), \text{ for } (f, g, \gamma) \in C^{n \geq 2}_\mathrm{rRBf} (A, M, \{ R_\alpha \}_{\alpha \in \Omega}).
\end{align*}
Explicitly, the map $\delta_\mathrm{rRBf} : C^n_\mathrm{rRBf} (A, M, \{ R_\alpha \}_{\alpha \in \Omega}) \rightarrow C^{n+1}_\mathrm{rRBf} (A, M, \{ R_\alpha \}_{\alpha \in \Omega})$ is given by
\begin{align*}
\delta_\mathrm{rRBf} \big(  (f,g) \big) =~& \big( \delta_\mathrm{Hoch} (f), \delta_\mathrm{Hoch}^f (g) , h_R (f,g)   \big), \text{ for } (f,g) \in C^1_\mathrm{rRBFf} (A, M, \{ R_\alpha \}_{\alpha \in \Omega}), \\
\delta_\mathrm{rRBf} \big(  (f, g, \gamma) \big) =~& \big( \delta_\mathrm{Hoch} (f), \delta_\mathrm{Hoch}^f (g) , \delta_R (\gamma) + h_R (f,g)   \big), \text{ for } (f, g, \gamma) \in C^{n \geq 2}_\mathrm{rRBFf} (A, M, \{ R_\alpha \}_{\alpha \in \Omega}).
\end{align*}
In the above expressions, $\delta_\mathrm{Hoch}$ is the Hochschild coboundary map of the algebra $A$ with coefficients in the adjoint $A$-bimodule. For any $f \in \mathrm{Hom}(A^{\otimes n}, A)$, the map $\delta^f_\mathrm{Hoch} : \mathrm{Hom}(\mathcal{A}^{n-1,1}, M ) \rightarrow \mathrm{Hom}(\mathcal{A}^{n,1}, M)$ is given by
\begin{align*}
\big(  \delta^f_\mathrm{Hoch} (g)  \big) (a_1, \ldots, a_{n+1}) =~& (l+r) \big(  a_1, (f+g) (a_2, \ldots, a_{n+1})   \big) \\
&+ \sum_{i=1}^n (-1)^i ~ g \big( a_1, \ldots, (\mu + l + r) (a_i, a_{i+1}), \ldots, a_{n+1} \big) \\
&+ (-1)^{n+1} (l+r) \big(  (f+g) (a_1, \ldots, a_n) , a_{n+1} \big),
\end{align*}
for $g \in \mathrm{Hom}(\mathcal{A}^{n-1,1}, M)$ and $a_1 \otimes \cdots \otimes a_{n+1} \in \mathcal{A}^{n,1}$. Finally, $\delta_R : \mathrm{Hom}_\Omega (M^{\otimes n}, A) \rightarrow \mathrm{Hom}_\Omega (M^{\otimes n+1}, A)$ is the differential induced by the collection of operators $R = \{ R_\alpha \}_{\alpha \in \Omega}$ (see subsection \ref{subsection-a graded lie}), and
\begin{align*}
h_R : \mathrm{Hom} (A^{\otimes n}, A) \oplus \mathrm{Hom} (\mathcal{A}^{n-1,1}, M) \rightarrow \mathrm{Hom}_\Omega (M^{\otimes n}, A)
\end{align*}
is given by
\begin{align*}
(h_R (f,g)&)_{\alpha_1, \ldots, \alpha_n} (u_1, \ldots, u_n) \\
&= (-1)^n \big\{ f \big(  R_{\alpha_1} (u_1) , \ldots, R_{\alpha_n} (u_n)  \big) - \sum_{r=1}^n R_{\alpha_1 \cdots \alpha_n} \big( g ( R_{\alpha_1} (u_1), \ldots, u_r, \ldots, R_{\alpha_n} (u_n) )  \big)     \big\}.
\end{align*}

\noindent Note that, upto some sign, the map $\delta_\mathrm{rRBf}$ is same with the map $l_1^{(s^{-1} \triangle, R)}$. Hence it follows that $(\delta_\mathrm{rRBf})^2 = 0$. In other words, $\{  C^\bullet_\mathrm{rRBf} (A,M, \{ R_\alpha \}_{\alpha \in \Omega}) , \delta_\mathrm{rRBf} \}$ is a cochain complex.

Let $Z^n_\mathrm{rRBf}  (A,M, \{ R_\alpha \}_{\alpha \in \Omega})   $ be the set of all $n$-cocycles and $B^n_\mathrm{rRBf}  (A,M, \{ R_\alpha \}_{\alpha \in \Omega})$ be the set of all $n$-coboundaries. Then we have $B^n_\mathrm{rRBf}  (A,M, \{ R_\alpha \}_{\alpha \in \Omega})  ~  \subset   ~Z^n_\mathrm{rRBf}  (A,M, \{ R_\alpha \}_{\alpha \in \Omega})   $, for all $n \geq 0$.  The quotients
\begin{align*}
H^n_\mathrm{rRBf}  (A,M, \{ R_\alpha \}_{\alpha \in \Omega})  := \frac{ Z^n_\mathrm{rRBf}  (A,M, \{ R_\alpha \}_{\alpha \in \Omega})   }{  B^n_\mathrm{rRBf}  (A,M, \{ R_\alpha \}_{\alpha \in \Omega})}, \text{ for } n \geq 0
\end{align*}
are called the {\bf cohomology groups} of the relative Rota-Baxter family algebra $(A, M, \{ R_\alpha \}_{\alpha \in \Omega})$.

\medskip

\subsection{Cohomology of Rota-Baxter family algebras} Using the general case of relative Rota-Baxter family algebras, here we define the cohomology of a Rota-Baxter family algebra.

Let $(A, \{ R_\alpha \}_{\alpha \in \Omega})$ be a Rota-Baxter family algebra. For each $n \geq 0$, we define the space of $n$-cochains $C^n_\mathrm{RBf} (A, \{ R_\alpha \}_{\alpha \in \Omega})$ by
\begin{align*}
   C^n_\mathrm{RBf} (A, \{ R_\alpha \}_{\alpha \in \Omega}) = \begin{cases}
        0 & \text{ if } n =0, \\
        \mathrm{Hom}(A,A) & \text{ if } n =1,\\
        \mathrm{Hom}(A^{\otimes n}, A) \oplus \mathrm{Hom}_\Omega (A^{\otimes n-1}, A) & \text{ if } n \geq 2.
    \end{cases}
\end{align*}
Then there is an embedding $i : C^\bullet_\mathrm{RBf} (A, \{ R_\alpha \}_{\alpha \in \Omega}) \hookrightarrow  C^\bullet_\mathrm{rRBf} (A, A,  \{ R_\alpha \}_{\alpha \in \Omega})$ given by
\begin{align*}
i (f ) :=~& (f,f), \text{ for } f \in C^1_\mathrm{RBf} (A, \{ R_\alpha \}_{\alpha \in \Omega}),\\
i(f , \gamma) :=~& (f, f, \gamma), \text{ for } (f, \gamma) \in C^{n \geq 2}_\mathrm{RBf} (A, \{ R_\alpha \}_{\alpha \in \Omega}).
\end{align*}
Here we consider $(A, A, \{ R_\alpha \}_{\alpha \in \Omega})$ as a relative Rota-Baxter family algebra and $C^n_\mathrm{rRBf} (A, A, \{ R_\alpha \}_{\alpha \in \Omega})$ is the $n$-th cochain group of the relative Rota-Baxter family algebra $(A, A, \{ R_\alpha \}_{\alpha \in \Omega})$, where $A$ is equipped with the adjoint $A$-bimodule structure.

It is easy to see that the differential $\delta_\mathrm{rRBf} : C^\bullet_\mathrm{rRBf} (A, A, \{ R_\alpha \}_{\alpha \in \Omega}) \rightarrow C^{\bullet +1}_\mathrm{rRBf} (A, A, \{ R_\alpha \}_{\alpha \in \Omega})$ satisfies
\begin{align*}
\delta_\mathrm{rRBf} \big( \mathrm{im} (i) \big) \subset \mathrm{im}(i).
\end{align*}
Hence $\delta_\mathrm{rRBf}$ induces a differential (denoted by $\delta_\mathrm{RBf}$) on the embedded subspace $C^\bullet_\mathrm{RBf} (A, \{ R_\alpha \}_{\alpha \in \Omega})$. The differential $\delta_\mathrm{RBf} : C^\bullet_\mathrm{RBf} (A, \{ R_\alpha \}_{\alpha \in \Omega}) \rightarrow C^{\bullet +1 }_\mathrm{RBf} (A, \{ R_\alpha \}_{\alpha \in \Omega})$ is given by
\begin{align*}
\delta_\mathrm{RBf} (f) =~& \big(  \delta_\mathrm{Hoch}(f), h_R (f)  \big), \text{ for } f \in \mathrm{Hom}(A,A) = C^1_\mathrm{RBf} (A, \{ R_\alpha \}_{\alpha \in \Omega}), \\
\delta_\mathrm{RBf} \big( (f, \gamma ) \big) =~& \big( \delta_\mathrm{Hoch} (f), \delta_R (\gamma) + h_R (f)   \big), \text{ for } (f, \gamma) \in C^{n \geq 2}_\mathrm{RBf} (A, \{ R_\alpha \}_{\alpha \in \Omega}).
\end{align*}
Here the map $h_R : \mathrm{Hom} (A^{\otimes n}, A) \rightarrow \mathrm{Hom}_\Omega (A^{\otimes n}, A)$ is given by
\begin{align*}
(h_R (f)&)_{\alpha_1, \ldots, \alpha_n} (a_1, \ldots, a_n) \\
&= (-1)^n \big\{ f \big(  R_{\alpha_1} (a_1) , \ldots, R_{\alpha_n} (a_n)  \big) - \sum_{r=1}^n R_{\alpha_1 \cdots \alpha_n} \big( f ( R_{\alpha_1} (a_1), \ldots, a_r, \ldots, R_{\alpha_n} (a_n) )  \big)     \big\},
\end{align*}
and $\delta_R : \mathrm{Hom}_\Omega (A^{\otimes n}, A) \rightarrow \mathrm{Hom}_\Omega (A^{\otimes n+1}, A) $  is the differential induced by the collection $R= \{ R_\alpha \}_{\alpha \in \Omega}$.
The cohomology groups of the cochain complex $\{ C^\bullet_\mathrm{RBf} (A, \{ R_\alpha \}_{\alpha \in \Omega}), \delta_\mathrm{RBf} \}$ are called the {\bf cohomology} of the Rota-Baxter family algebra $(A, \{ R_\alpha \}_{\alpha \in \Omega})$.

\medskip

Let $(A, \{ R_\alpha \}_{\alpha \in \Omega})$ be a Rota-Baxter family algebra. Then we have the following cochain complexes:

\medskip

$\diamond$ The Hochschild cochain complex $\{ C^\bullet_\mathrm{Hoch} (A,A), \delta_\mathrm{Hoch} \}$ of the associative algebra $A$ with coefficients in the adjoint $A$-bimodule.

\medskip

$\diamond$ The cochain complex $\{ \mathrm{Hom}_\Omega (A^{\otimes \bullet}, A), d_R \}$ induced by the collection of operators $ R = \{ R_\alpha \}_{\alpha \in \Omega}$.

\medskip

$\diamond$ The cochain complex $\{ C^\bullet_\mathrm{RBf} (A, \{ R_\alpha \}_{\alpha \in \Omega} ) , \delta_\mathrm{RBf} \}$ of the Rota-Baxter family algebra $(A, \{ R_\alpha \}_{\alpha \in \Omega})$.

\medskip 

\noindent It is easy to see that there is a short exact sequence of cochain complexes
\begin{align*}
0 \rightarrow \{    \mathrm{Hom}_\Omega (A^{\otimes \bullet -1}, A), d_R \} \xrightarrow{i} \{   C^\bullet_\mathrm{RBf} (A, \{ R_\alpha \}_{\alpha \in \Omega}), \delta_\mathrm{RBf}  \}  \xrightarrow{p} \{  C^\bullet_\mathrm{Hoch} (A,A), \delta_\mathrm{Hoch} \}  \rightarrow 0,
\end{align*}
where $i (\gamma) = (0, \gamma)$ and $p (f, \gamma) = f$. As a consequence, we obtain the following.

\begin{thm}
Let $(A, \{ R_\alpha \}_{\alpha \in \Omega})$ be a Rota-Baxter family algebra. Then there is a long exact sequence
\begin{align*}
 \cdots \longrightarrow H^{n-1}_R (A,A) \longrightarrow H^n_\mathrm{RBf} (A, \{ R_\alpha \}_{\alpha \in \Omega}) \longrightarrow H^n_\mathrm{Hoch} (A,A) \longrightarrow H^n_R (A,A) \longrightarrow \cdots .
\end{align*}
\end{thm}

\medskip

\section{Deformations of relative Rota-Baxter family algebras}\label{sec5}
In this section, we study formal and infinitesimal deformations of a relative Rota-Baxter family algebra. In particular, we show that the set of all equivalence classes of infinitesimal deformations of a relative Rota-Baxter family algebra $(A, M, \{ R_\alpha \}_{\alpha \in \Omega})$ has a bijection with the second cohomology group $H^2_\mathrm{rRBf} (A, M, \{ R_\alpha \}_{\alpha \in \Omega}).$

\medskip

Let $(A, M, \{ R_\alpha \}_{\alpha \in \Omega})$ be a relative Rota-Baxter family algebra. We denote the associative multiplication on $A$ by the map $\mu$, and the maps defining the left and right $A$-actions on $M$ by $l$ and $r$, respectively. Let $A[[t]]$ (resp. $M[[t]]$) be the space of all formal power series in $t$ with coefficients from $A$ (resp. $M$). Then $A[[t]]$ and $M[[t]]$ are both ${\bf k}[[t]]$-modules.

\begin{defn}
A {\bf formal one-parameter deformation} of a relative Rota-Baxter family algebra $(A, M, \{ R_\alpha \}_{\alpha \in \Omega})$ is a quadruple $(\mu_t, l_t, r_t, \{ (R_t)_\alpha \}_{\alpha \in \Omega})$  that consists of three formal power series of the form
\begin{align*}
\mu_t = \sum_{i \geq 0} t^i \mu_i, \quad l_t = \sum_{i \geq 0} t^i l_i, \quad r_t = \sum_{i \geq 0} t^i r_i
\end{align*}
and  a collection of power series $\{ (R_t)_\alpha = \sum_{i \geq 0} t^i R_{i, \alpha} \}_{\alpha \in \Omega}$ labelled by the elements of $\Omega$ (where each $\mu_i \in \mathrm{Hom} (A \otimes A, A)$, $l_i \in \mathrm{Hom}(A \otimes M, M)$, $r_i \in \mathrm{Hom} (M \otimes A, M)$ and $R_{i, \alpha} \in \mathrm{Hom}(M,A)$ with $\mu_0 = \mu$, $l_0 = l$, $r_0 = r$ and $R_{0, \alpha} = R_{\alpha}$, for all $\alpha \in \Omega$) such that

\medskip

$\diamond$ the ${\bf k} [[t]]$-linear operation $\mu_t : A[[t]] \otimes_{{\bf k}[[t]]} A[[t]] \rightarrow A[[t]]$ makes the pair $(A[[t]], \mu_t)$ into an associative algebra over the ring ${\bf k}[[t]]$,

\medskip

$\diamond$ the ${\bf k} [[t]]$-linear operations $l_t : A[[t]] \otimes_{{\bf k}[[t]]} M[[t]] \rightarrow M[[t]]$ and $r_t : M[[t]] \otimes_{{\bf k}[[t]]} A[[t]] \rightarrow M[[t]]$ makes the triple $(M[[t]], l_t, r_t)$ into a bimodule over the associative algebra $(A[[t]], \mu_t)$,

\medskip

$\diamond$ the collection of ${\bf k} [[t]]$-linear maps $\{ (R_t)_\alpha : M[[t]] \rightarrow A[[t]] \}_{\alpha \in \Omega}$ labelled by the elements of $\Omega$ makes the triple $\big( (A[[t]], \mu_t),  (M[[t]], l_t, r_t), \{ (R_t)_\alpha \}_{\alpha \in \Omega} \big)$ into a relative Rota-Baxter family algebra over the base ring ${\bf k}[[t]].$
\end{defn}

It follows from the above definition that a quadruple $(\mu_t, l_t, r_t,   \{ (R_t)_\alpha \}_{\alpha \in \Omega}   )$ is a formal one-parameter deformation if and only if the following system of equations must hold: For all $n \geq 0$, $a, b, c \in A$, $u, v \in M$ and $\alpha, \beta \in \Omega$,
\begin{align}
\sum_{i+j=n} \mu_i(\mu_j(a,b),c) &=  \displaystyle\sum_{i+j=n} \mu_i(a,\mu_j(b,c)), \label{def-eqn1}\\
\sum_{i+j=n} l_i(\mu_j(a,b),u) &= \displaystyle\sum_{i+j=n} l_i(a,l_j(b,u)),\\
\sum_{i+j=n} r_i(l_j(a,u),b) &= \displaystyle\sum_{i+j=n} l_i(a,r_j(u,b),\\
\sum_{i+j=n} r_i(r_j(u,a),b)&=\displaystyle\sum_{i+j=n} r_i(u,\mu_j(a,b)),\\
\sum_{i+j+k = n} \mu_i \big(  R_{j, \alpha} (u), R_{k, \beta}(v)  \big) &= \sum_{i+j+k = n} R_{i, \alpha \beta} \big( l_j (   R_{k, \alpha} (u), v) + r_j (u, R_{k, \beta} (v))   \big).  \label{def-eqn5}
\end{align}
These equations are called the {\bf deformation equations}. Note that the Equations (\ref{def-eqn1})-(\ref{def-eqn5}) are hold for $n=0$ as $(A,M, \{ R_\alpha \}_{\alpha \in \Omega})$ is a relative Rota-Baxter family algebra. If we write the above equations for $n=1$, we obtain
\begin{align}
\mu_1(a\cdot b,c) + \mu_1(a,b)\cdot c &= \mu_1(a,b\cdot c) + a\cdot\mu_1(b,c),\label{def-eqn11}\\
l_1(a\cdot b,u) + \mu_1(a,b)\cdot_M u &= l_1(a,b\cdot_M u) + a\cdot_M l_1(b,u),\label{def-eqn22}\\
r_1(a\cdot_M u,b) + l_1(a,u)\cdot_M b &=  l_1(a,u\cdot_M b) + a\cdot_M r_1(u,b),\label{def-eqn33}\\
r_1(u\cdot_M a,b) + r_1(u,a)\cdot_M b &= r_1(u,a\cdot b) + u\cdot_M \mu_1(a,b),\label{def-eqn44}\\
R_{1, \alpha}(u) \cdot R_\beta (v) + R_\alpha (u) \cdot R_{1, \beta} (v) + \mu_1 (R_\alpha (u), R_\beta (v)) &= R_{\alpha \beta} \big( l_1 (R_\alpha (u), v) + R_{1, \alpha} (u) \cdot_M v  \big) \label{def-eqn55} \\
 + R_{1, \alpha \beta} \big(  R_\alpha (u) \cdot_M v  \big) + R_{\alpha \beta} \big(  r_1 ~& (u, R_\beta (v)) + u \cdot_M R_{1, \beta} (v)  \big) + R_{1, \alpha \beta} (u \cdot_M R_\beta (v)),  \nonumber
\end{align}
for all $a, b, c \in A$, $u, v \in M$ and $\alpha, \beta \in \Omega$. The identity (\ref{def-eqn11}) is equivalent to $(\delta_\mathrm{Hoch} (\mu_1)) (a,b,c) = 0$. To understand the remaining identities, we first define an element $\beta_1 \in \mathrm{Hom}(\mathcal{A}^{1,1}, M)$ by
\begin{align}\label{beta}
\beta_1 (a, u) := l_1 (a, u) ~~~~ \text{ and } ~~~~ \beta_1 (u, a) := r_1 (u, a), \text{ for } a \in A, u \in M.
\end{align}
Then the identities (\ref{def-eqn22}), (\ref{def-eqn33}), (\ref{def-eqn44}) are respectively equivalent to the followings:
\begin{align*}
\big( \delta^{\mu_1}_\mathrm{Hoch} (\beta_1)  \big) (a,b, u) = 0, \quad  \big( \delta^{\mu_1}_\mathrm{Hoch} (\beta_1)  \big) (a, u, b) = 0 \quad \text{ and } \quad \big( \delta^{\mu_1}_\mathrm{Hoch} (\beta_1)  \big) (u, a, b) = 0.
\end{align*}
Combining these, we get $\delta^{\mu_1}_\mathrm{Hoch} (\beta_1) = 0$. Finally, the identity (\ref{def-eqn55}) is equivalent to the condition
\begin{align*}
\big( \delta_R (R_1) + h_R (\mu_1, \beta_1)  \big)_{\alpha, \beta} (u, v) = 0,
\end{align*}
where $R_1 = \{ R_{1, \alpha} \}_{\alpha \in \Omega}$. Hence we have
\begin{align*}
\delta_{\mathrm{rRBf}} (\mu_1, \beta_1, R_1) = \big( \delta_\mathrm{Hoch} (\mu_1), ~ \delta^{\mu_1}_\mathrm{Hoch} (\beta_1), ~  \delta_R (R_1) + h_R (\mu_1, \beta_1)  \big) = 0.
\end{align*}
Therefore, the element $(\mu_1, \beta_1, R_1) \in Z^2_\mathrm{rRBf} (A, M, \{ R_\alpha \}_{\alpha \in \Omega})$ is a $2$-cocycle. This is called the {\bf infinitesimal} of the given formal one-parameter deformation $(\mu_t, l_t, r_t, \{ (R_t)_\alpha \}_{\alpha \in \Omega})$.

\begin{defn}
Let $(\mu_t, l_t, r_t, \{ (R_t)_\alpha \}_{\alpha \in \Omega})$ and $(\mu'_t, l'_t, r'_t, \{ (R'_t)_\alpha \}_{\alpha \in \Omega})$ be two formal one-parameter deformations of a relative Rota-Baxter family algebra $(A, M, \{ R_\alpha \}_{\alpha \in \Omega})$. They are said to be {\bf equivalent} if there exists a pair $(\varphi_t, \psi_t)$ of formal power series of the form
\begin{align}\label{formal-form}
\varphi_t = \sum_{i \geq 0} t^i \varphi_i   \quad \text{ and } \quad \psi_t = \sum_{i \geq 0} t^i \psi_i
\end{align}
(where each $\varphi_i \in \mathrm{Hom}(A,A)$ and $\psi_i \in \mathrm{Hom}(A,A)$ with $\varphi_0 = \mathrm{id}_A$ and $\psi_0 = \mathrm{id}_M$) such that the pair of ${\bf k}[[t]]$-linear maps
\begin{align*}
(\varphi_t, \psi_t) : \big(  (A[[t]], \mu_t), (M[[t]], l_t, r_t), \{ (R_t)_\alpha \}_{\alpha \in \Omega} \big) \rightsquigarrow \big(  (A[[t]], \mu'_t), (M[[t]], l'_t, r'_t), \{ (R'_t)_\alpha \}_{\alpha \in \Omega} \big)
\end{align*}
is a morphism of relative Rota-Baxter family algebras.
\end{defn}

It follows that the deformations $(\mu_t, l_t, r_t, \{ (R_t)_\alpha \}_{\alpha \in \Omega})$ and $(\mu'_t, l'_t, r'_t, \{ (R'_t)_\alpha \}_{\alpha \in \Omega})$ are equivalent if there exists a pair $(\varphi_t, \psi_t)$ of formal power series of the form (\ref{formal-form}) such that the following system of equations are held: for all $n \geq 0$, $a, b \in A$, $u \in M$ and $\alpha \in \Omega$,
\begin{align}
\sum_{i+j=n} \varphi_i(\mu_j(a,b)) &=\displaystyle\sum_{i+j+k=n}\mu'_i(\varphi_j(a),\varphi_k(b)), \label{equiv-def-eqn1}\\
\sum_{i+j=n} \psi_i (l_j(a,u)) &=\displaystyle\sum_{i+j+k=n}l'_i(\varphi_j(a),\psi_j(u)), \label{equiv-def-eqn2}\\
\sum_{i+j=n} \psi_i (r_j(u,a)) &=\displaystyle\sum_{i+j+k=n}r'_i(\psi_j(u),\varphi_k(a)), \label{equiv-def-eqn3}\\
\sum_{i+j = n} \varphi_i \circ R_{j, \alpha} &= \sum_{i+j= n} R'_{i, \alpha} \circ \psi_j, \label{equiv-def-eqn4}
\end{align}
It is easy to see that all the above identities are held for $n=0$ as $\varphi_0 = \mathrm{id}_A$ and $\psi_0 = \mathrm{id}_M$. By writing the Equation (\ref{equiv-def-eqn1}) for $n=1$, we simply get
\begin{align*}
\mu_1 - \mu_1' = \delta_\mathrm{Hoch} (\varphi_1).
\end{align*}
Similarly, the Equations (\ref{equiv-def-eqn2})-(\ref{equiv-def-eqn3}) yields $\beta_1 - \beta_1' = \delta_\mathrm{Hoch}^{\varphi_1} (\psi_1)$. Finally, the Equation (\ref{equiv-def-eqn4}) for $n=1$ gives rise to
\begin{align*}
R_1 - R_1' = h_R (\varphi_1, \psi_1).
\end{align*}
Thus, we get
\begin{align*}
(\mu_1, \beta_1, R_1) - (\mu_1', \beta_1', R_1') = \big( \delta_\mathrm{Hoch} (\varphi_1), ~ \delta_\mathrm{Hoch}^{\varphi_1} (\psi_1), ~  h_R (\varphi_1, \psi_1) \big) = \delta_\mathrm{rRBf} (\varphi_1, \psi_1).
\end{align*}
In conclusion, we have the following result.

\begin{thm}
Let $(A, M, \{ R_\alpha \}_{\alpha \in \Omega})$ be a relatibe Rota-Baxter family algebra. Then the infinitesimal of any formal one-parameter deformation of the relative Rota-Baxter family algebra $(A, M, \{ R_\alpha \}_{\alpha \in \Omega})$ is a $2$-cocycle in the cochain complex $\{ C^\bullet_\mathrm{rRBf} (A, M, \{ R_\alpha \}_{\alpha \in \Omega}), \delta_\mathrm{rRBf}  \}$. Moreover, the corresponding cohomology class in $H^2_\mathrm{rRBf} (A, M, \{ R_\alpha \}_{\alpha \in \Omega}) $ depends only on the equivalence class of the deformation. Hence there is a well-defined map
\begin{align}\label{defor-map}
(\text{set of all formal one-parameter deformations})/ \sim  ~ \longrightarrow ~ H^2_\mathrm{rRBf} (A, M, \{ R_\alpha \}_{\alpha \in \Omega}).
\end{align}
\end{thm}

\medskip

It follows from the above theorem that an equivalence class of formal one-parameter deformations give rise to a cohomology class in $H^2_\mathrm{rRBf} (A, M, \{ R_\alpha \}_{\alpha \in \Omega})$. However, the converse need not be true. In the following, we consider infinitesimal deformations of a relative Rota-Baxter family algebra $(A, M, \{ R_\alpha \}_{\alpha \in \Omega})$ and show that any cohomology class in $H^2_\mathrm{rRBf} (A, M, \{ R_\alpha \}_{\alpha \in \Omega})$ yields an infinitesimal deformation. In particular, the map (\ref{defor-map}) has a generalization while considering infinitesimal deformations (instead of formal one-parameter deformations) that turns out to be bijective.

\begin{defn}
Let $(A, M, \{ R_\alpha \}_{\alpha \in \Omega})$ be a relative Rota-Baxter family algebra. An {\bf infinitesimal deformation} of $(A, M, \{ R_\alpha \}_{\alpha \in \Omega})$ is a quadruple $(\mu_t, l_t, r_t, \{ (R_t)_\alpha \}_{\alpha \in \Omega})$ of the following truncated sums
\begin{align*}
\mu_t = \mu+ t \mu_1, \quad l_t = l + t l_1, \quad r_t = r+t r_1 ~~~ \text{ and } ~~~ \{ (R_t)_\alpha = R_\alpha + t R_{1, \alpha} \}_{\alpha \in \Omega}
\end{align*}
such that the triple $\big(  (A[[t]]/(t^2), \mu_t), (M[[t]]/(t^2), l_t, r_t), \{ (R_t)_\alpha \}_{\alpha \in \Omega} \big)$ is a relative Rota-Baxter family algebra over the ring ${\bf k}[[t]]/(t^2)$.
\end{defn}

Equivalence between two infinitesimal deformations can be defined similarly as we defined the same for formal one-parameter deformations. We are now ready to prove the main result of this section.

\begin{thm}
Let $(A, M, \{ R_\alpha \}_{\alpha \in \Omega})$ be a relative Rota-Baxter family algebra. Then there is a bijection between the set of all equivalence classes of infinitesimal deformations of $(A, M, \{ R_\alpha \}_{\alpha \in \Omega})$ and the second cohomology group $H^2_\mathrm{rRBf} (A, M, \{ R_\alpha \}_{\alpha \in \Omega}).$
\end{thm}

\begin{proof}
Like formal one-parameter deformations, we have a map from the set of all infinitesimal deformations to the second cohomology group $H^2_\mathrm{rRBf} (A,M, \{ R_\alpha \}_{\alpha \in \Omega}).$ We will now construct a map in the other direction.

\medskip

Let $(\mu_1, \beta_1, R_1) \in Z^2_\mathrm{rRBf} (A,M, \{ R_\alpha \}_{\alpha \in \Omega})$ be a $2$-cocycle. Then it is easy to verify that the quadruple $(\mu_t, l_t, r_r, \{ (R_t)_\alpha \}_{\alpha \in \Omega})$ is an infinitesimal deformation, where
\begin{align*}
\mu_t = \mu + t \mu_1, \quad l_t = l +t l_1, \quad r_t = r + t r_1 ~~~~ \text{ and } ~~~~ \{ (R_t)_\alpha = R_\alpha + t R_{1, \alpha} \}_{\alpha \in \Omega}.
\end{align*} 
Here the maps $l_1$ and $r_1$ are obtained from $\beta_1$ simply by (\ref{beta}). If $(\mu_1, \beta_1, R_1)$ and $(\mu'_1, \beta'_1, R'_1)$ are two cohomologous $2$-cocycles \big(say, $(\mu_1, \beta_1, R_1) - (\mu_1, \beta_1, R_1) = \delta_\mathrm{rRBf} (\varphi_1, \psi_1)$\big) then the pair 
\begin{align*}
(\varphi_t = \mathrm{id}_A + t \varphi_1, \psi_t = \mathrm{id}_M + t \psi_1)
\end{align*}
 provides the equivalence between the corresponding infinitesimal deformations. Hence we obtain a map from $H^2_\mathrm{rRBf} (A, M, \{ R_\alpha \}_{\alpha \in \Omega})$ to the set of all equivalence classes of infinitesimal deformations. This map is the inverse of the previously constructed map in the other direction. Hence the result follows.
\end{proof}

\medskip

\section{Homotopy Rota-Baxter family algebras and Dend$_\infty$-family algebras}\label{sec6} In this section, we first introduce homotopy (relative) Rota-Baxter family algebras as the homotopy analogue of (relative) Rota-Baxter family algebras. Next, we also consider Dend$_\infty$-family algebras (homotopy analogue of dendriform family algebras) and find their relations with homotopy (relative) Rota-Baxter family algebras. Like relative Rota-Baxter family algebras have underlying associative algebras, their homotopy version has underlying $A_\infty$-algebras. Thus, we start this section by recalling $A_\infty$-algebras and representations over them \cite{stasheff,keller}.

\medskip

\begin{defn}\label{a-inf-defn}
    An {\bf $A_\infty$-algebra} is a pair $(A, \{ \mu_k \}_{k \geq 1})$ consisting of a graded vector space $A$ with a collection $\{ \mu_k : A^{\otimes k} \rightarrow A \}_{k \geq 1}$ of degree $1$ linear maps satisfying
    \begin{align}\label{homotopy-ass-iden}
\sum_{k+l = n+1} \sum_{i=1}^k (-1)^{|a_1| + \cdots + |a_{i-1}|} ~ \mu_k \big( a_1, \ldots, a_{i-1}, \mu_l (a_i, \ldots, a_{i+l-1}), a_{i+l}, \ldots, a_n   \big) = 0,
    \end{align}
    for any $n \geq 1$ and homogeneous elements $a_1, \ldots, a_n \in A.$
\end{defn}

Let $(A, \{ \mu_k \}_{k \geq 1})$ be an $A_\infty$-algebra. A {\bf representation} of this $A_\infty$-algebra is a graded vector space $M$ equipped with a collection $\{ \eta_k : \mathcal{A}^{k-1,1} \xrightarrow{} M \}_{k \geq 1}$ of degree $1$ linear maps satisfying the identities (\ref{homotopy-ass-iden}) with exactly one of $a_1, \ldots, a_n$ is from $M$ and the corresponding operation $\mu_k$ or $\mu_l$ replaced by $\eta_k$ or $\eta_l$.
Like the nonhomotopic case, here $\mathcal{A}^{k-1,1}$ denotes the direct sum of all possible tensor powers of $A$ and $M$ in which $A$ appears $k-1$ times and $M$ appears exactly once.

Representations of $A_\infty$-algebras generalize the concept of bimodules over associative algebras. Further, it follows from the above definition that any $A_\infty$-algebra is a representation of itself, called the adjoint representation.

\subsection{Homotopy (relative) Rota-Baxter family algebras} Here we introduce homotopy (relative) Rota-Baxter family algebras. To do this, we will mainly generalize the results of subsection \ref{subsection-a graded lie} in the homotopy context.

Let $(A, \{ \mu_k \}_{k \geq 1})$ be an $A_\infty$-algebra and $(M, \{ \eta_k \}_{k \geq 1})$ be a representation of it. 
For any $n \in \mathbb{Z}$ and $k \geq 1$, let $\mathrm{Hom}^n_\Omega ( (A \oplus M)^{\otimes k}, A \oplus M)$ be the set whose elements are given by a collection 
\begin{align}\label{mu-k}
\mu_k = \big\{  (\mu_k)_{\alpha_1, \ldots, \alpha_k} : (A \oplus M)^{\otimes k} \rightarrow A \oplus M  \big\}_{\alpha_1, \ldots, \alpha_k \in \Omega}
\end{align}
 of degree $n$ linear maps labelled by the elements of $\Omega^{\times k}$. We define
\begin{align*}
\mathrm{Hom}_\Omega^n \big(   \overline{T}(A \oplus M), A \oplus M \big) = \oplus_{k \geq 1} \mathrm{Hom}^n_\Omega ( (A \oplus M)^{\otimes k}, A \oplus M).
\end{align*}
Thus, an element $\mu \in \mathrm{Hom}^n_\Omega ( \overline{T} (A \oplus M), A \oplus M)$ is given by a sum $\mu = \sum_{k \geq 1} \mu_k$, where  $\mu_k$'s are of the form (\ref{mu-k}).
Observe that the graded vector space $\oplus_{n \in \mathbb{Z}} \mathrm{Hom}^n_\Omega (\overline{T}(A \oplus M), A \oplus M)$ inherits a graded Lie bracket
\begin{align*}
    \{ \! \! [ \sum_{k \geq 1} \mu_k , \sum_{l \geq 1} \eta_l  ] \! \! \}_\Omega = \sum_{p \geq 1} \sum_{k+l = p+1} \big( \mu_k \circ \eta_l - (-1)^{mn} \eta_l \circ \mu_k  \big),
\end{align*}
where
\begin{align*}
&(\mu_k \circ \eta_l)_{\alpha_1, \ldots, \alpha_{p}} \big(  (a_1, u_1), \ldots, (a_{p}, u_{p})  \big)  \\
&=  \sum_{i=1}^k (-1)^{|a_1| + \cdots + |a_{i-1}|} 
 (\mu_k)_{\alpha_1, \ldots, \alpha_i \cdots \alpha_{i+l-1}, \ldots, \alpha_{k+l-1}} \big(  (a_1, u_1), \ldots, (a_{i-1}, u_{i-1}), \\ & \qquad \qquad \qquad \qquad \qquad \qquad \qquad  (\eta_l)_{\alpha_i, \ldots, \alpha_{i+l-1}} \big(   (a_i, u_i), \ldots, (a_{i+l-1}, u_{i+l-1})  \big), \ldots,  (a_{p}, u_{p})      \big),
 \end{align*}
for $\sum_{k \geq 1} \mu_k \in \mathrm{Hom}^m_\Omega ( \overline{T} (A \oplus M), A \oplus M)$ and $\sum_{l \geq 1} \eta_l \in \mathrm{Hom}^n_\Omega ( \overline{T} (A \oplus M), A \oplus M)$.
It is further easy to see that the graded vector space $\mathfrak{a} = \oplus_{n \in \mathbb{Z}} \mathrm{Hom}_\Omega^n (\overline{T}(M), A)$ is an abelian Lie subalgebra of $\mathfrak{g} = \big(  \oplus_{n \in \mathbb{Z}} \mathrm{Hom}^n_\Omega (\overline{T}(A \oplus M), A \oplus M), \{ \! \! [ ~, ~ ] \! \! \}_\Omega  \big)$. Let $P: \mathfrak{g} \rightarrow \mathfrak{g}$ be the projection onto the subspace $\mathfrak{a}$. We define an element $\triangle \in   \mathrm{Hom}^1_\Omega     \big(  \overline{T}(A \oplus M), A \oplus M  \big)$ given by
\begin{align*}
\triangle := \widetilde{\sum_{k \geq 1} (\mu_k + \eta_k )}.
\end{align*}
Then the element $\triangle \in \mathrm{ker}(P)_1$. Since $(A, \{ \mu_k \}_{k \geq 1})$ is an $A_\infty$-algebra and $(M, \{ \eta_k \}_{k \geq 1})$ is a representation, it follows that $\{ \! \! [ \triangle , \triangle ] \! \! \}_\Omega = 0$. Hence we obtain a $V$-data $(\mathfrak{g}, \mathfrak{a}, P, \triangle)$. Thus, by Theorem \ref{cons-thm} (a), the graded vector space $\mathfrak{a} =  \oplus_{n \in \mathbb{Z}} \mathrm{Hom}_\Omega^n (\overline{T}(M), A)$ inherits an $L_\infty$-algebra with structure maps $\{ l_k \}_{k \geq 1}$. (Note that, this is in general a genuine $L_\infty$-algebra which generalizes the graded Lie algebra of Theorem \ref{mc-gla-thm} in the homotopy context.) 

\begin{defn}\label{defn-aaa}
A {\bf homotopy relative Rota-Baxter family algebra} is a triple
\begin{align}\label{hrrbfa}
\big(  (A, \{ \mu_k \}_{k \geq 1}), (M, \{ \eta_k \}_{k \geq 1}), R = \sum_{k \geq 1} R_k \big)
\end{align}
in which $(A, \{ \mu_k \}_{k \geq 1})$ is an $A_\infty$-algebra, $(M, \{ \eta_k \}_{k \geq 1})$ is a representation and $R = \sum_{k \geq 1} R_k$ is a Maurer-Cartan element of the $L_\infty$-algebra $\big(  \oplus_{n \in \mathbb{Z}} \mathrm{Hom}_\Omega^n (\overline{T}(M), A), \{ l_k \}_{k \geq 1} \big)$.
\end{defn}

The equivalent (and explicit) description of a homotopy relative Rota-Baxter family algebra is given by the following result.

\begin{thm}\label{thm-aaa}
A homotopy relative Rota-Baxter family algebra is a triple (\ref{hrrbfa}) in which $(A, \{ \mu_k \}_{k \geq 1})$ is an $A_\infty$-algebra, $(M, \{ \eta_k \}_{k \geq 1})$ is a representation and $R = \sum_{k \geq 1} R_k \in \mathrm{Hom}^0_\Omega (\overline{T}(M), A)$ satisfying the following set of identities
\begin{align}\label{homotopy-rota-iden}
&\sum_{t_1 + \cdots + t_k = n} \frac{1}{k!} ~ \mu_k \big( (R_{t_1})_{\alpha_1, \ldots, \alpha_{t_1}} (u_1, \ldots, u_{t_1}), \ldots, (R_{t_k})_{ \alpha_{t_1 + \cdots + t_{k-1} + 1} , \ldots, \alpha_n} (u_{t_1 + \cdots + t_{k-1} + 1} , \ldots, u_n)   \big)  \\
&= \sum_{i=1}^k \sum_{\substack{j+t_1 + \cdots t_{i-1} + 1\\ + t_{i+1} + \cdots t_k + t = n}} \frac{(-1)^{|u_1| + \cdots + |u_j|}}{(k-1 )!} ~ (R_{j+1+t})_{\alpha_1, \ldots, \alpha_j, \alpha_{j+1} \cdots \alpha_{n-t}, \alpha_{n-t+1}, \ldots, \alpha_n} \bigg( u_1, \ldots, u_j, \nonumber \\
&\eta_k \big( (R_{t_1})_{\alpha_{j+1}, \ldots, \alpha_{j+ t_1}} (u_{j+1}, \ldots, u_{j+ t_1}), \ldots, u_{j+ t_1 + \cdots + t_{i-1} + 1}, \ldots, (R_{t_k})_{\alpha_{n - t - t_k +1}, \ldots, \alpha_{n-t}} (u_{n - t - t_k +1}, \ldots, u_{n-t}) \big), \nonumber \\ & u_{n-t+1}, \ldots, u_n \bigg), \text{ for all } n \geq 1, ~\alpha_1, \ldots, \alpha_n \in \Omega \text{ and } u_1, \ldots, u_n \in M. \nonumber
\end{align}
\end{thm}

\begin{proof}
Note that the element $R$ is a Maurer-Cartan element if and only if
\begin{align*}
\sum_{k=1}^\infty \frac{1}{k!}~ P \underbrace{[ \cdots [  [}_{k} \triangle, R ], R ], \ldots, R ] ~= 0.
\end{align*}
For any $k \geq 1$, we have
\begin{align}\label{something-eqn}
&\big(  P \underbrace{[ \cdots [  [}_{k} \triangle, R ], R ], \ldots, R ] \big)  \nonumber \\
&= \mathrm{pr}_A \sum_{i=0}^k (-1)^i ~ {k \choose i} \big( \underbrace{R \circ \cdots \circ R}_{i} \circ \triangle \underbrace{\circ R \circ \cdots \circ R}_{k-i}  \big) \nonumber \\
&=\big(  ( \mathrm{pr}_A \circ \triangle \circ \underbrace{R \circ \cdots \circ R }_{k} ) - k ( \mathrm{pr}_A \circ R \circ  \triangle \circ \underbrace{R \circ \cdots \circ R }_{k-1}) \big) ~ (\text{the other terms are zero}).
\end{align}
By careful calculation, we get
\begin{align*}
&( \mathrm{pr}_A \circ \triangle \circ \underbrace{R \circ \cdots \circ R }_{k} ) (u_1, \ldots, u_n) \\
&= \sum_{t_1 + \cdots + t_k = n}  \mu_k \big( (R_{t_1})_{\alpha_1, \ldots, \alpha_{t_1}} (u_1, \ldots, u_{t_1}), \ldots, (R_{t_k})_{ \alpha_{t_1 + \cdots + t_{k-1} + 1} , \ldots, \alpha_n} (u_{t_1 + \cdots + t_{k-1} + 1} , \ldots, u_n)   \big).
\end{align*}
Similarly,
\begin{align*}
&k ( \mathrm{pr}_A \circ R \circ  \triangle \circ \underbrace{R \circ \cdots \circ R }_{k-1}) (u_1, \ldots, u_n) \\
& = k \sum_{i=1}^k \sum_{\substack{j+t_1 + \cdots t_{i-1} + 1\\ + t_{i+1} + \cdots t_k + t = n}} (-1)^{|u_1| + \cdots + |u_j|} ~ (R_{j+1+t})_{\alpha_1, \ldots, \alpha_j, \alpha_{j+1} \cdots \alpha_{n-t}, \alpha_{n-t+1}, \ldots, \alpha_n} \bigg( u_1, \ldots, u_j, \nonumber \\
&\eta_k \big( (R_{t_1})_{\alpha_{j+1}, \ldots, \alpha_{j+ t_1}} (u_{j+1}, \ldots, u_{j+ t_1}), \ldots, u_{j+ t_1 + \cdots + t_{i-1} + 1}, \ldots, (R_{t_k})_{\alpha_{n - t - t_k +1}, \ldots, \alpha_{n-t}} (u_{n - t - t_k +1}, \ldots, u_{n-t}) \big), \nonumber \\ & u_{n-t+1}, \ldots, u_n \bigg).
\end{align*}
Therefore, it follows from (\ref{something-eqn}) that the identities (\ref{homotopy-rota-iden}) hold. This completes the proof.
\end{proof}

A homotopy relative Rota-Baxter family algebra $\big(  (A, \{ \mu_k \}_{k \geq 1}), (M, \{ \eta_k \}_{k \geq 1}), R = \sum_{k \geq 1} R_k \big)$ is said to be {\bf strict} if $R_k = 0$ for $k \neq 1$. Therefore, $R$ is given by a collection $R = \{ R_\alpha : M \rightarrow A \}_{\alpha \in \Omega}$ of degree $0$ linear maps. Moreover, it follows from (\ref{homotopy-rota-iden}) that the collection $R = \{ R_\alpha \}_{\alpha \in \Omega}$ satisfies
\begin{align}\label{strict-iden}
\mu_k \big( R_{\alpha_1} (u_1), \ldots, R_{\alpha_k} (u_k)  \big) = \sum_{r=1}^k  R_{\alpha_1 \cdots \alpha_k} \big(  \eta_k (   R_{\alpha_1} (u_1), \ldots, u_r, \ldots, R_{\alpha_k} (u_k)) \big),
\end{align}
for all $k \geq 1$, $\alpha_1, \ldots, \alpha_k \in \Omega$ and $u_1, \ldots, u_k \in M$.

\medskip

\begin{remark}\label{hrbfal} In Definition \ref{defn-aaa} or in Theorem \ref{thm-aaa}, if we consider the adjoint representation of an $A_\infty$-algebra, we obtain homotopy Rota-Baxter family algebra. 
\end{remark}


It follows from Theorem \ref{thm-aaa} and Remark \ref{hrbfal} that any (relative) Rota-Baxter family algebra can be seen as a homotopy (relative) Rota-Baxter family algebra concentrated in degree $-1$.

\medskip

\subsection{Dend$_\infty$-family algebras} In this subsection, we introduce and study Dend$_\infty$-family algebras as the homotopy version of dendriform family algebras.

Let $D = \oplus_{i \in \mathbb{Z}} D_i$ be a graded vector space. For each $n \geq 1$, let $C_n$ be the set of first $n$ natural numbers. We write $C_n = \{ [1], \ldots, [n] \}$ for our convenience. For any $n \in \mathbb{Z}$ and $k geq 1$, let
 $\mathrm{Hom}_\Omega^n ({\bf k} [C_k] \otimes D^{\otimes k}, D)$ be the set whose elements are given by a collection $\theta_k = \{ (\theta_k)_{\alpha_1, \ldots, \alpha_k} : {\bf k} [C_k] \otimes D^{\otimes k} \rightarrow D \}_{\alpha_1, \ldots, \alpha_k \in \Omega}$ of graded linear maps labelled by the elements of $\Omega^{\times k}$. For any $[r] \in C_k$, we denote the degree $n$ linear map $(\theta_k)_{\alpha_1, \ldots, \alpha_k} ([r]; ~, \ldots, ~) : D^{\otimes k} \rightarrow D$ simply by $(\theta_k)^{[r]}_{\alpha_1, \ldots, \alpha_k}$. 

Let $\mathrm{Hom}_{\widehat{\Omega}}^n ({\bf k} [C_k] \otimes D^{\otimes k}, D) \subset \mathrm{Hom}_\Omega^n ({\bf k} [C_k] \otimes D^{\otimes k}, D)$ be the set consisting of all elements 
\begin{align*}
\theta_k = \{ (\theta_k)_{\alpha_1, \ldots, \alpha_k} : {\bf k} [C_k] \otimes D^{\otimes k} \rightarrow D \}_{\alpha_1, \ldots, \alpha_k \in \Omega} \in \mathrm{Hom}_\Omega^n ({\bf k} [C_k] \otimes D^{\otimes k}, D)
\end{align*}
for which the degree $n$ linear maps $(\theta_k)^{[r]}_{\alpha_1, \ldots, \alpha_k} : D^{\otimes k} \rightarrow D$ doesn't depend on $\alpha_r$, for all $1 \leq r \leq k$.

\begin{defn}
A {\bf Dend$_\infty$-family algebra} is a pair $(D, \{ \theta_k \}_{k \geq 1} )$ consisting of a graded vector space $D = \oplus_{i \in \mathbb{Z}} D_i$ with a collection $\{ \theta_k \}_{k \geq 1}$ of elements $\theta_k \in \mathrm{Hom}^1_{\widehat{\Omega}}({\bf k} [C_k] \otimes D^{\otimes k}, D)$, for $k \geq 1$, satisfying the following set of identities: for any $ n \geq 1,$ $[r] \in C_n$, $ \alpha_1, \ldots, \alpha_n \in \Omega$ and homogeneous elements $x_1, \ldots, x_n \in D$, 
\begin{align}\label{hom-dend-fam-iden}
 \sum_{k+l = n+1} \sum_{i = 1}^k ~(\theta_k \diamond_i \theta_l)^{[r]}_{\alpha_1, \ldots, \alpha_n} (x_1, \ldots, x_n) = 0,
\end{align}
where\\

\noindent $(\theta_k \diamond_i \theta_l)^{[r]}_{\alpha_1, \ldots, \alpha_n} (x_1, \ldots, x_n) = $

\begin{align*}
\begin{cases}
(-1)^{|x_1| + \cdots + |x_{i-1}|} ~ (\theta_k)^{[r]}_{\alpha_1, \ldots, \alpha_i \cdots \alpha_{i+l-1}, \ldots, \alpha_n} \big( x_1, \ldots, x_{i-1}, \sum_{s=1}^l (\theta_l)^{[s]}_{\alpha_i, \ldots, \alpha_{i+l-1}} (x_i, \ldots, x_{i+l-1}), x_{i+l}, \ldots, x_{n}    \big) \\
\hfill \text{ if }  1 \leq r \leq i-1, \\\\
(-1)^{|x_1| + \cdots + |x_{i-1}|} ~ (\theta_k)^{[i]}_{\alpha_1, \ldots, \alpha_i \cdots \alpha_{i+l-1}, \ldots, \alpha_n} \big( x_1, \ldots, x_{i-1},  (\theta_l)^{[r-i+1]}_{\alpha_i, \ldots, \alpha_{i+l-1}} (x_i, \ldots, x_{i+l-1}), x_{i+l}, \ldots, x_{n}    \big) \\
\hfill \text{ if } i \leq r \leq i+l-1, \\\\
(-1)^{|x_1| + \cdots + |x_{i-1}|} ~ (\theta_k)^{[r-l+1]}_{\alpha_1, \ldots, \alpha_i \cdots \alpha_{i+l-1}, \ldots, \alpha_n} \big( x_1, \ldots, x_{i-1}, \sum_{s=1}^l (\theta_l)^{[s]}_{\alpha_i, \ldots, \alpha_{i+l-1}} (x_i, \ldots, x_{i+l-1}), x_{i+l}, \ldots, x_{n}    \big) \\
\hfill \text{ if } i+l \leq r \leq n. 
\end{cases}
\end{align*}
\end{defn}

\medskip

\begin{remark}
(i) When $\Omega$ is singleton, a Dend$_\infty$-family algebra is nothing but a Dend$_\infty$-algebra (homotopy dendriform algebra) \cite{lod-val-book}. Thus, Dend$_\infty$-family algebras are a generalization of Dend$_\infty$-algebras.

\medskip

(ii) Note that any dendriform family algebra can be regarded as a Dend$_\infty$-family algebra concentrated in degree $-1$. More precisely, if $(D, \{ \prec_\alpha, \succ_\alpha \}_{\alpha \in \Omega})$ is a dendriform family algebra then $(s^{-1} D, \{ \theta_k \}_{k \geq 1})$ is a Dend$_\infty$-family algebra, where
\begin{align*}
\begin{cases}
(\theta_2)^{[1]}_{\alpha, \beta} (s^{-1}x, s^{-1}y) = s^{-1} (x \prec_\beta y) \\
(\theta_2)^{[2]}_{\alpha, \beta} (s^{-1}x, s^{-1} y) = s^{-1} (x \succ_\alpha y)
\end{cases}   ~~~~ \text{ and } \quad  \theta_k = 0 \text{ for } k \neq 2.
\end{align*}
\end{remark}

It is known that Dend$_\infty$-algebras are the splitting object for $A_\infty$-algebras \cite{lod-val-book}. Here we generalize this result by introducing a notion of $\Omega$-$A_\infty$-algebra.

\begin{defn}
An {\bf $\Omega$-$A_\infty$-algebra} (also called an $A_\infty$-algebra relative to $\Omega$) is a pair $(A, \{ \nu_k \}_{k \geq 1})$ consisting of a graded vector space $A$ together with a collection $\{ \nu_k \}_{k \geq 1}$ of elements $\nu_k \in \mathrm{Hom}^1_\Omega (A^{\otimes k}, A)$, for $k \geq 1$, satisfying
\begin{align*}
\sum_{k+l = n+1} \sum_{i=1}^k (-1)^{|a_1| + \cdots + |a_{i-1}|} ~ (\mu_k)_{\alpha_1, \ldots, \alpha_i \cdots \alpha_{i+l-1}, \ldots, \alpha_{n}} \big( a_1, \ldots, (\mu_l)_{\alpha_i, \ldots, \alpha_{i+l-1}} (a_i, \ldots, a_{i+l-1}), \ldots, a_n   \big) = 0,
\end{align*}
for all $n \geq 1$, $\alpha_1, \ldots, \alpha_n \in \Omega$ and homogeneous elements $a_1, \ldots, a_n \in A$. 
\end{defn}

When $\Omega$ is singleton, an $\Omega$-$A_\infty$-algebra is nothing but an ordinary $A_\infty$-algebra (see Definition \ref{a-inf-defn}). In general, we have the following result.

\begin{prop}\label{omega-to-ordinary}
Let $(A, \{ \nu_k \}_{k \geq 1})$ be an $\Omega$-$A_\infty$-algebra. Then $(A \otimes {\bf k} \Omega, \{ \mu_k \}_{k \geq 1})$ is an ordinary $A_\infty$-algebra, where
\begin{align*}
\mu_k (a_1 \otimes \alpha_1, \ldots, a_k \otimes \alpha_k) = \big( (\nu_k)_{\alpha_1, \ldots, \alpha_k} (a_1, \ldots, a_k)   \big) \otimes \alpha_1 \cdots \alpha_k,
\end{align*}
for $a_1 \otimes \alpha_1, \ldots, a_k \otimes \alpha_k \in A \otimes {\bf k} \Omega$.
\end{prop}

\begin{proof}
For any $n \geq 1$ and homogeneous elements $a_1 \otimes \alpha_1, \ldots, a_n \otimes \alpha_n \in A \otimes {\bf k} \Omega$, we have
\begin{align*}
&\sum_{k+l = n+1} \sum_{i=1}^k (-1)^{|a_1 \otimes \alpha_1| +  \cdots + |a_{i-1} \otimes \alpha_{i-1}|} \\
& \qquad \qquad \qquad  \mu_k \big(  a_1 \otimes \alpha_1, \ldots, a_{i-1} \otimes \alpha_{i-1}, \mu_l \big(  a_i \otimes \alpha_i, \ldots, a_{i+l-1} \otimes \alpha_{i+l-1}  \big), a_{i+l} \otimes \alpha_{i+l} , \ldots, a_n \otimes \alpha_n  \big) \\
&= \sum_{k+l = n+1} \sum_{i=1}^k (-1)^{|a_1| + \cdots + |a_{i-1}|} ~ (\mu_k)_{\alpha_1, \ldots, \alpha_i \cdots \alpha_{i+l-1}, \ldots, \alpha_{n}} \big( a_1, \ldots, (\mu_l)_{\alpha_i, \ldots, \alpha_{i+l-1}} (a_i, \ldots, a_{i+l-1}), \ldots, a_n   \big) \\
&= 0.
\end{align*}
This completes the proof.
\end{proof}

\begin{prop}
Let $(D, \{ \theta_k \}_{k \geq 1})$ be a Dend$_\infty$-family algebra. Then $(D, \{ \nu_k \}_{k \geq 1})$ is an $\Omega$-$A_\infty$-algebra, where
\begin{align*}
(\nu_k)_{\alpha_1, \ldots, \alpha_k} (x_1, \ldots, x_k) = \sum_{r=1}^k (\theta_k)^{[r]}_{\alpha_1, \ldots, \alpha_k} (x_1, \ldots, x_k),
\end{align*}
for $\alpha_1, \ldots, \alpha_k \in \Omega$ and $x_1, \ldots, x_k \in \Omega$.
\end{prop}

\begin{proof}
Since $(D, \{ \theta_k \}_{k \geq 1})$ is a Dend$_\infty$-family algebra, we have the identities (\ref{hom-dend-fam-iden}). Adding all these identities for $r=1, 2, \ldots, n$, we get
\begin{align}\label{r1n}
\sum_{r=1}^n ~ \sum_{k+l= n+1} ~ \sum_{i=1}^k ~(\theta_k \diamond_i \theta_l)^{[r]}_{\alpha_1, \ldots, \alpha_n} (x_1, \ldots, x_n) = 0.
\end{align}
Given any fixed $k, l$ (with $k+l=n+1$) and $1 \leq i \leq k$, we have $\sum_{r=1}^n = \sum_{r=1}^{i-1} + \sum_{r=i}^{i+l-1} + \sum_{r=i+l}^n$. For this decomposition,
\begin{align}
&\sum_{r=1}^{i-1} (\theta_k  \diamond_i \theta_l)^{[r]}_{\alpha_1, \ldots, \alpha_n} (x_1, \ldots, x_n) \nonumber \\
&= \sum_{r=1}^{i-1} (-1)^{|x_1| + \cdots + |x_{i-1}|}~ (\theta_k)^{[r]}_{\alpha_1, \ldots, \alpha_i \cdots \alpha_{i+l-1}, \ldots, \alpha_n} \big(  x_1, \ldots, x_{i-1}, \sum_{s=1}^l (\theta_l)^{[s]}_{\alpha_i, \ldots, \alpha_{i+l-1}} (x_i, \ldots, x_{i+l-1}), \ldots, x_n  \big) \nonumber \\
&= \sum_{r=1}^{i-1} (-1)^{|x_1| + \cdots + |x_{i-1}|}~ (\theta_k)^{[r]}_{\alpha_1, \ldots, \alpha_i \cdots \alpha_{i+l-1}, \ldots, \alpha_n} \big(  x_1, \ldots, x_{i-1}, (\nu_l)_{\alpha_i, \ldots, \alpha_{i+l-1}} (x_i, \ldots, x_{i+l-1}), \ldots, x_n  \big). \label{some-identity1}
\end{align}
Similarly,
\begin{align}
&\sum_{r=i}^{i+l-1} (\theta_k  \diamond_i \theta_l)^{[r]}_{\alpha_1, \ldots, \alpha_n} (x_1, \ldots, x_n) \nonumber \\
&= (-1)^{|x_1| + \cdots + |x_{i-1}|}~ (\theta_k)^{[i]}_{\alpha_1, \ldots, \alpha_i \cdots \alpha_{i+l-1}, \ldots, \alpha_n} \big(    x_1, \ldots, x_{i-1}, \sum_{s=1}^l (\theta_l)^{[s]}_{\alpha_i, \ldots, \alpha_{i+l-1}} (x_i, \ldots, x_{i+l-1}), \ldots, x_n \big)    \nonumber \\
&= (-1)^{|x_1| + \cdots + |x_{i-1}|}~ (\theta_k)^{[i]}_{\alpha_1, \ldots, \alpha_i \cdots \alpha_{i+l-1}, \ldots, \alpha_n} \big(    x_1, \ldots, x_{i-1}, (\nu_l)_{\alpha_i, \ldots, \alpha_{i+l-1}} (x_i, \ldots, x_{i+l-1}), \ldots, x_n \big)    \label{some-identity2}
\end{align}
and
\begin{align}
&\sum_{r=i+l}^{n} (\theta_k  \diamond_i \theta_l)^{[r]}_{\alpha_1, \ldots, \alpha_n} (x_1, \ldots, x_n) \nonumber \\
&= \sum_{r=i+1}^{k} (-1)^{|x_1| + \cdots + |x_{i-1}|}~ (\theta_k)^{[r]}_{\alpha_1, \ldots, \alpha_i \cdots \alpha_{i+l-1}, \ldots, \alpha_n} \big(  x_1, \ldots, x_{i-1}, \sum_{s=1}^l (\theta_l)^{[s]}_{\alpha_i, \ldots, \alpha_{i+l-1}} (x_i, \ldots, x_{i+l-1}), \ldots, x_n  \big) \nonumber \\
&= \sum_{r=i+1}^{k} (-1)^{|x_1| + \cdots + |x_{i-1}|}~ (\theta_k)^{[r]}_{\alpha_1, \ldots, \alpha_i \cdots \alpha_{i+l-1}, \ldots, \alpha_n} \big(  x_1, \ldots, x_{i-1}, (\nu_l)_{\alpha_i, \ldots, \alpha_{i+l-1}} (x_i, \ldots, x_{i+l-1}), \ldots, x_n  \big).
 \label{some-identity3}
\end{align}
Substituting the values of (\ref{some-identity1}), (\ref{some-identity2}) and (\ref{some-identity3}) in the identity (\ref{r1n}), we get that $(D, \{ \nu_k \}_{k \geq 1})$ is an $\Omega$-$A_\infty$-algebra.
\end{proof}

In the following, we find some relations between strict homotopy relative Rota-Baxter family algebras and Dend$_\infty$-family algebras.

\begin{thm}
Let $\big( (A, \{ \mu_k \}_{k \geq 1}), (M, \{ \eta_k \}_{k \geq 1}), \{ R_\alpha \}_{\alpha \in \Omega} \big)$ be a strict homotopy relative Rota-Baxter family algebra. Then $(M, \{ \theta_k \}_{k \geq 1})$ is a Dend$_\infty$-family algebra, where
\begin{align*}
(\theta_k)^{[r]}_{\alpha_1, \ldots, \alpha_k} (u_1, \ldots, u_k) = \eta_k \big(   R_{\alpha_1} (u_1), \ldots, u_r, \ldots, R_{\alpha_k} (u_k) \big),
\end{align*}
for $k \geq 1$, $[r] \in C_k$, $\alpha_1, \ldots, \alpha_k \in \Omega$ and $u_1, \ldots, u_k \in M.$
\end{thm}

\begin{proof}
Since the collection $\{ R_\alpha \}_{\alpha \in \Omega}$ satisfies (\ref{strict-iden}), it follows that
\begin{align}
\mu_k \big(  R_{\alpha_1} (u_1) , \ldots, R_{\alpha_k} (u_k) \big) = R_{\alpha_1 \cdots \alpha_k} \big(  \sum_{r=1}^k (\theta_k)^{[r]}_{\alpha_1, \ldots, \alpha_k} (u_1, \ldots, u_k)   \big),
\end{align}
for all $k \geq 1$, $\alpha_1, \ldots, \alpha_k \in \Omega$ and $u_1, \ldots, u_k \in M$.  We consider the elements $R_{\alpha_1} (u_1), \ldots, u_r, \ldots, R_{\alpha_n} (u_n)$ in which exactly one element $u_r$ is from $M$ and all the remaining elements are from $A$. Let $k, l, i$ be fixed natural numbers with $k+l=n+1$ and $1 \leq i \leq k$. If $1 \leq r \leq i-1$, then
\begin{align*}
&(-1)^{|u_1| + \cdots + |u_{i-1}|} ~ \eta_k \big(  R_{\alpha_1} (u_1), \ldots, u_r, \ldots, R_{\alpha_{i-1}} (u_{i-1}) , \mu_l \big( R_{\alpha_i} (u_i), \ldots, R_{\alpha_{i+l-1}} (u_{i+l-1})   \big), \ldots, R_{\alpha_n} (u_n)  \big) \\
&= (-1)^{|u_1| + \cdots + |u_{i-1}|} ~ \eta_k \big(  R_{\alpha_1} (u_1), \ldots, u_r, \ldots, R_{\alpha_i \cdots \alpha_{i+l-1}} \big( \sum_{s=1}^l (\theta_l)^{[s]}_{\alpha_i , \ldots, \alpha_{i+l-1}} (u_i, \ldots, u_{i+l-1})  \big), \ldots, R_{\alpha_n} (u_n) \big)  \\
&= (-1)^{|u_1| + \cdots + |u_{i-1}|} ~ (\theta_k)^{[r]}_{\alpha_1, \ldots, \alpha_i \cdots \alpha_{i+l-1}, \ldots, \alpha_n} \big( u_1, \ldots, \sum_{s=1}^l (\theta_l)^{[s]}_{\alpha_i, \ldots, \alpha_{i+l-1}} (u_i, \ldots, u_{i+l-1}), \ldots, u_n   \big) \\
&= (\theta_k \diamond_i \theta_l)^{[r]}_{\alpha_1, \ldots, \alpha_n} (u_1, \ldots, u_n).
\end{align*}
Similarly, if $i \leq r \leq i+l-1$ then
\begin{align*}
&(-1)^{|u_1| + \cdots + |u_{i-1}|} ~ \eta_k \big(  R_{\alpha_1} (u_1), \ldots, R_{\alpha_{i-1}} (u_{i-1}) , \eta_l \big( R_{\alpha_i} (u_i), \ldots, u_r, \ldots, R_{\alpha_{i+l-1}} (u_{i+l-1})   \big), \ldots, R_{\alpha_n} (u_n)  \big) \\
&= (-1)^{|u_1| + \cdots + |u_{i-1}|} ~ \eta_k \big(  R_{\alpha_1} (u_1), \ldots, (\theta_l)^{[r-i+1]}_{\alpha_i, \ldots, \alpha_{i+l-1}} (u_i, \ldots, u_{i+l-1}), \ldots, R_{\alpha_n} (u_n) \big)  \\
&= (-1)^{|u_1| + \cdots + |u_{i-1}|} ~ (\theta_k)^{[i]}_{\alpha_1, \ldots, \alpha_i \cdots \alpha_{i+l-1}, \ldots, \alpha_n} \big( u_1, \ldots, u_{i-1}, (\theta_l)^{[r-i+1]}_{\alpha_i, \ldots, \alpha_{i+l-1}} (u_i, \ldots, u_{i+l-1}), \ldots, u_n   \big) \\
&= (\theta_k \diamond_i \theta_l)^{[r]}_{\alpha_1, \ldots, \alpha_n} (u_1, \ldots, u_n)
\end{align*}
and if $i + l \leq r \leq n$, then
\begin{align*}
&(-1)^{|u_1| + \cdots + |u_{i-1}|} ~ \eta_k \big(  R_{\alpha_1} (u_1),  \ldots, R_{\alpha_{i-1}} (u_{i-1}) , \mu_l \big( R_{\alpha_i} (u_i), \ldots, R_{\alpha_{i+l-1}} (u_{i+l-1})   \big), \ldots, u_r, \ldots, R_{\alpha_n} (u_n)  \big) \\
&= (-1)^{|u_1| + \cdots + |u_{i-1}|} ~ \eta_k \big(  R_{\alpha_1} (u_1),  \ldots, R_{\alpha_i \cdots \alpha_{i+l-1}} \big( \sum_{s=1}^l (\theta_l)^{[s]}_{\alpha_i , \ldots, \alpha_{i+l-1}} (u_i, \ldots, u_{i+l-1})  \big), \ldots, u_r, \ldots, R_{\alpha_n} (u_n) \big)  \\
&= (-1)^{|u_1| + \cdots + |u_{i-1}|} ~ (\theta_k)^{[r-l+1]}_{\alpha_1, \ldots, \alpha_i \cdots \alpha_{i+l-1}, \ldots, \alpha_n} \big( u_1, \ldots, u_{i-1}, \sum_{s=1}^l (\theta_l)^{[s]}_{\alpha_i, \ldots, \alpha_{i+l-1}} (u_i, \ldots, u_{i+l-1}), \ldots, u_n   \big) \\
&= (\theta_k \diamond_i \theta_l)^{[r]}_{\alpha_1, \ldots, \alpha_n} (u_1, \ldots, u_n).
\end{align*}
Therefore, in any case, we obtain the identities (\ref{hom-dend-fam-iden}) as the pair $(M, \{ \eta_k \}_{k \geq 1})$ is a representation of the $A_\infty$-algebra $(A, \{ \mu_k \}_{k \geq 1})$.
\end{proof}

Next, we show that a Dend$_\infty$-family algebra has an associated homotopy relative Rota-Baxter family algebra. Moreover, the induced Dend$_\infty$-family algebra structure coincides with the given one.

\begin{thm}
(i) Let $(D, \{ \theta_k \}_{k \geq 1})$ be a Dend$_\infty$-family algebra. Then $(D \otimes {\bf k} \Omega, \{ \overline{\theta}_k \}_{k \geq 1})$ is a Dend$_\infty$-algebra, where
\begin{align*}
(\overline{\theta}_k)^{[r]} \big( x_1 \otimes \alpha_1, \ldots, x_k \otimes \alpha_k \big) := (\theta_k)^{[r]}_{\alpha_1, \ldots, \alpha_k} (x_1, \ldots, x_k) \otimes \alpha_1 \cdots \alpha_k,
\end{align*}
for $x_1 \otimes \alpha_1, \ldots, x_k \otimes \alpha_k \in D \otimes {\bf k} \Omega$ and $[r] \in C_k$.

(ii) The pair $(D \otimes {\bf k} \Omega, \{ \mu_k \}_{k \geq 1})$ is an $A_\infty$-algebra, where
\begin{align*}
\mu_k \big(  x_1 \otimes \alpha_1, \ldots, x_k \otimes \alpha_k \big) := \sum_{r=1}^k (\theta_k)^{[r]}_{\alpha_1, \ldots, \alpha_k} (x_1, \ldots, x_k) \otimes \alpha_1 \cdots \alpha_k, \text{ for } x_i \otimes \alpha_i \in D \otimes {\bf k} \Omega.
\end{align*}
Moreover, $(D, \{ \eta_k \}_{k \geq 1})$ is a representation of the $A_\infty$-algebra $(D \otimes {\bf k} \Omega, \{ \mu_k \}_{k \geq 1})$, where
\begin{align*}
\eta_k \big(  x_1 \otimes \alpha_1, \ldots, x_r , \ldots, x_k \otimes \alpha_k \big) := (\theta_k)^{[r]}_{\alpha_1, \ldots, \alpha_k} (x_1, \ldots, x_k),
\end{align*}
for $x_i \otimes \alpha_i \in D \otimes {\bf k} \Omega$ $(i=1, \ldots, r-1, r+1, \ldots, k)$ and $x_r \in D$. On the right-hand side of the above defining identity, we take $\alpha_r$ to be any element of $\Omega$.

(iii) The triple $\big(  (D \otimes {\bf k} \Omega , \{ \mu_k \}_{k \geq 1}), (D, \{ \eta_k \}_{k \geq 1}), \{ R_\alpha \}_{\alpha \in \Omega}  \big)$ is a strict homotopy relative Rota-Baxter family algebra, where
\begin{align*}
R_\alpha : D \rightarrow D \otimes {\bf k} \Omega,~ R_\alpha (x) = x \otimes \alpha, ~ \text{ for } \alpha \in \Omega, x \in D.
\end{align*}
Further, the induced Dend$_\infty$-family algebra structure on $D$ coincides with the given one.
\end{thm}

\begin{proof}
(i) The proof is similar to the proof of Proposition \ref{omega-to-ordinary}. Hence we will not repeat it here.

(ii) Since $(D \otimes {\bf k} \Omega , \{ \overline{\theta}_k \}_{k \geq 1}  )$ is a Dend$_\infty$-algebra, it follows that $(D \otimes {\bf k} \Omega, \{ \mu_k \}_{k \geq 1})$ is an $A_\infty$-algebra, where
\begin{align*}
\mu_k = \sum_{r=1}^k (\overline{\theta}_k)^{[r]}, \text{ for } k \geq 1. 
\end{align*}
This proves the first part. The second part is easy to verify as the conditions for representation are equivalent to the Dend$_\infty$-family algebra conditions (\ref{hom-dend-fam-iden}).

(iii) For any $\alpha_1, \ldots, \alpha_k \in \Omega$ and $x_1, \ldots, x_k \in D$, we have
\begin{align*}
\mu_k \big(  R_{\alpha_1} (x_1), \ldots, R_{\alpha_k} (x_k) \big) =~& \mu_k \big( x_1 \otimes \alpha_1, \ldots, x_k \otimes \alpha_k   \big) \\
=~& \sum_{r=1}^k \big(  (\theta_k)^{[r]}_{\alpha_1, \ldots, \alpha_k} (x_1, \ldots, x_k)  \big) \otimes \alpha_1 \cdots \alpha_k \\
=~& \sum_{r=1}^k R_{\alpha_1 \cdots \alpha_k} \big(    (\theta_k)^{[r]}_{\alpha_1, \ldots, \alpha_k} (x_1, \ldots, x_k)    \big) \\
=~& \sum_{r=1}^k R_{\alpha_1 \cdots \alpha_k} \big(  \eta_k \big(  R_{\alpha_1} (x_1), \ldots, x_r, \ldots, R_{\alpha_k} (x_k)  \big)    \big).
\end{align*}
This completes the first part. Finally, if $(D, \{  {\theta}'_k  \}_{k \geq 1})$ is the induced Dend$_\infty$-family algebra, then
\begin{align*}
({\theta}'_k)^{[r]}_{\alpha_1, \ldots, \alpha_k} (x_1, \ldots, x_k) = \eta_k \big( R_{\alpha_1} (x_1), \ldots, x_r, \ldots, R_{\alpha_k} (x_k) \big) = (\theta_k)^{[r]}_{\alpha_1, \ldots, \alpha_k} (x_1, \ldots, x_k).
\end{align*}
Hence the result follows.
\end{proof}

\medskip

\noindent {\bf Acknowledgements.} 
The author would like to thank Indian Institute of Technology (IIT) Kharagpur for providing a beautiful academic environment where the research has been carried out.

\medskip

\noindent {\bf Data Availability Statement.} Data sharing is not applicable to this article as no new data were created or analyzed in this study.

\end{document}